\documentclass[a4paper,10pt,twoside]{article}
\usepackage[english]{babel}

\usepackage{psfrag}
\usepackage{amsmath, amsthm, amssymb, amsfonts}
\usepackage{graphicx}
\usepackage{rotating}
\usepackage{multirow}
\usepackage{multicol}
\usepackage{relsize}
\usepackage{geometry}
\usepackage{float}
\usepackage{mathtools}
\usepackage{multicol}

\newcommand{\Ed}{E_D}
\newcommand{\Pd}{P_D}
\newcommand{\mN}{\mathcal{N}}
\newcommand{\pO}{\partial \Omega}

\newcommand{\bilformj}[2]{a^{(j)} \left( #1, \ #2 \right)}
\newcommand{\bilformi}[2]{a_i \left( #1, \ #2 \right)}
\newcommand{\bilforme}[2]{a_e \left( #1, \ #2 \right)}
\newcommand{\bilformiloc}[2]{a_i ^{(j)} \left( #1, \ #2 \right)}
\newcommand{\bilformeloc}[2]{a_e ^{(j)} \left( #1, \ #2 \right)}
\newcommand{\Gammaj}{\Gamma^{(j)}}
\newcommand{\harmext}[2]{\mathcal{H}^\Delta_{#1} #2}
\newcommand{\harmextie}[2]{\mathcal{H}^{i,e}_{#1} #2}

\newcommand{\innerprod}[2]{\left( #1, \ #2 \right)}
\newcommand{\interpol}[1]{I^h(#1)}
\newcommand{\norm}[2]{|| #1 ||_{L^2(#2)} }
\newcommand{\normtau}[1]{||| #1 |||_{\tau} }
\newcommand{\seminormHone}[2]{| #1 |_{H^1(#2)} }
\newcommand{\seminormHhalf}[2]{| #1 |_{H^{1/2}(#2)} }
\newcommand{\seminormS}[1]{| #1 |_{S_\Gamma} }
\newcommand{\seminormSj}[1]{| #1 |_{S_\Gamma^{(j)}} }
\newcommand{\seminormSjF}[1]{| #1 |_{S_\mathcal{F}^{(j)}} }
\newcommand{\seminormSkF}[1]{| #1 |_{S_\mathcal{F}^{(k)}} }
\newcommand{\seminormSjuno}[1]{| #1 |_{S_\Gamma^{(j_1)}} }

\newcommand{\Iapp}{I_{\text{app}}}
\newcommand{\Ion}{I_{\text{ion}}}
\newcommand{\skk}{s^{k}}
\newcommand{\sikk}{s_i^{k}}
\newcommand{\sekk}{s_e^{k}}
\newcommand{\silkk}{s_{i,l}^{k}}
\newcommand{\selkk}{s_{e,l}^{k}}

\newcommand{\dvdt}{\dfrac{\partial v}{\partial t}}
\newcommand{\dwdt}{\dfrac{\partial w}{\partial t}}
\newcommand{\sumdidv}{\sum_{l=1}^{N_h} \dfrac{\partial I_{\text{ion}}}{\partial v_l} (v^{k-1})}
\newcommand{\sumdidve}{\sum_{l=1}^{N_h} \dfrac{\partial I_{\text{ion}}}{\partial v_l} (v)}
\newcommand{\pIon}{\dfrac{\partial I_{\text{ion}}}{\partial v_l} (v^k)}

\newcommand{\sumJN}[1]{\sum_{j = 1}^N #1}

\newcommand{\sumNx}[1]{\sum_{k \in \mathcal{N}_x} #1 }

\newcommand{\diej}{ \delta_j^{i,e \ \dagger}}

\newcommand{\RF}{R_{\mathcal{F}}}
\newcommand{\RE}{R_{\mathcal{E}}}

\newcommand{\SjF}{S_{\mathcal{F}}^{(j)}}
\newcommand{\SkF}{S_{\mathcal{F}}^{(k)}}
\newcommand{\Sj}{S^{(j)}}
\newcommand{\Sk}{S^{(k)}}
\newcommand{\SjallE}{S_{\mathcal{E}}^{(j_{123})}}
\newcommand{\SjunoE}{S_{\mathcal{E}}^{(j_1)}}
\newcommand{\SjdueE}{S_{\mathcal{E}}^{(j_2)}}
\newcommand{\SjtreE}{S_{\mathcal{E}}^{(j_3)}}

\newcommand{\uiejF}{u_{j, \mathcal{F}}^{i,e}}
\newcommand{\uiekF}{u_{k, \mathcal{F}}^{i,e}}

\newcommand{\uFj}{u_{j, \mathcal{F}}}
\newcommand{\uFk}{u_{k, \mathcal{F}}}

\newcommand{\uEj}{u_{j, \mathcal{E}}}
\newcommand{\uEjuno}{u_{j_1, \mathcal{E} } }
\newcommand{\uEjdue}{u_{j_2, \mathcal{E}}}
\newcommand{\uEjtre}{u_{j_3, \mathcal{E}}}
\newcommand{\umeanE}{\bar{u}_\mathcal{E} }
\newcommand{\uiemeanE}{\bar{u}^{i,e}_\mathcal{E} }

\newcommand{\uiemeanEj}{\bar{u}^{i,e}_{j,\mathcal{E} } }
\newcommand{\uiemeanEk}{\bar{u}^{i,e}_{k,\mathcal{E} } }
\newcommand{\uiemeanEjuno}{\bar{u}^{i,e}_{j_1,\mathcal{E} } }
\newcommand{\umeanF}{\bar{u}_\mathcal{F}}

\newcommand{\umeanieFj}{\bar{u}^{i,e}_{j,\mathcal{F}} }
\newcommand{\umeanieFk}{\bar{u}^{i,e}_{k,\mathcal{F}} }
\newcommand{\umeanFie}{\bar{u}_\mathcal{F}^{i,e}}

\newtheorem{theorem}{Theorem}[section]

\newtheorem{lemma}{Lemma}[section]
\newtheorem{proposition}{Proposition}[section]
\newtheorem{remark}{Remark}[section]

\title{Parallel Newton-Krylov-BDDC and FETI-DP deluxe solvers for implicit time discretizations of the cardiac Bidomain equations}

\author{
	Ngoc Mai Monica Huynh
	\thanks{Dipartimento di Matematica, Universit\`a degli Studi di Pavia, 
		Via Ferrata, 27100 Pavia, Italy.
		E-mail: {\sf ngocmaimonica.huynh@unipv.it},
		{\sf luca.pavarino@unipv.it}.
		This work was supported by grants of MIUR (PRIN 2017AXL54F\_002),
		Istituto Nazionale di Alta Matematica (INDAM-GNCS) and
		the EuroHPC Joint Undertaking under grant agreement No 955495 (MICROCARD).
	}, \
	Luca F. Pavarino \footnotemark[1] \
	and Simone Scacchi
	\thanks{Dipartimento di Matematica, Universit\`a degli Studi di Milano,
		Via Saldini 50, 20133 Milano, Italy.
		E-mail: {\sf simone.scacchi@unimi.it}.
		This work was supported by grants of MIUR (PRIN 2017AXL54F\_003),
		Istituto Nazionale di Alta Matematica (INDAM-GNCS) and
		the EuroHPC Joint Undertaking under grant agreement No 955495 (MICROCARD).}
}

\begin{document}
	
	\maketitle
	\begin{center}
		\today
	\end{center}

	\begin{abstract}
		Two novel parallel Newton-Krylov Balancing Domain Decomposition by Constraints (BDDC) and Dual-Primal Finite Element Tearing and Interconnecting (FETI-DP) solvers with deluxe scaling are constructed, analyzed and tested numerically for implicit time discretizations of the three-dimensional Bidomain system of equations. 
		This model represents the most advanced mathematical description of the cardiac bioelectrical activity and it consists of a degenerate system of two non-linear reaction-diffusion partial differential equations (PDEs), coupled with a stiff system of ordinary differential equations (ODEs). 
		A finite element discretization in space and a segregated implicit discretization in time, based on decoupling the PDEs from the ODEs, yields at each time step the solution of a non-linear algebraic system.
		The Jacobian linear system at each Newton iteration is solved by a Krylov method, accelerated by BDDC or FETI-DP preconditioners, both augmented with the recently introduced {\em deluxe} scaling of the dual variables.
		A polylogarithmic convergence rate bound is proven for the resulting parallel Bidomain solvers. 
		Extensive numerical experiments on Linux clusters up to two thousands processors confirm the theoretical estimates, showing that the proposed parallel solvers are scalable and quasi-optimal. 
	\end{abstract}
	
	\pagestyle{myheadings}
	\thispagestyle{plain}
	\markboth{NMM Huynh, LF Pavarino and S Scacchi}{Newton-Krylov-Dual-Primal Bidomain solvers}
	
	\section{Introduction}
	The aim of this work is to design, analyze and numerically test Balancing Domain Decomposition by Constraints (BDDC) and Finite Element Tearing and Interconnecting dual-primal (FETI-DP) algorithms for a preconditioned Newton-Krylov solver for implicit time discretizations of the Bidomain system.
	This model describes the propagation of the electric impulse in cardiac tissue by means of a degenerate parabolic system of two non-linear reaction-diffusion partial differential equations (PDEs), modelling the evolution of the transmembrane electric potential \cite{franzone2014mathematical, pennacchio2005multiscale, quarteroni2017cardiovascular, quarteroni2017integrated}. These PDEs are coupled through the non-linear reaction term with a system of ordinary differential equations (ODEs), deriving from a membrane model describing the ionic currents flowing through the cell membrane and the dynamics of the associated gating variables.
	
	Several time discretizations have been proposed for these complex nonlinear cardiac models, employing different semi-implicit, operator splitting and decoupling techniques; see \cite[Ch. 7.2]{franzone2014mathematical} for a review.
	The most popular time discretizations have been based on semi-implicit (see e.g. \cite{charawi2017, colli2018numerical, franzone2015parallel, zampini2014dual, zampini2014inexact}) and/or operator splitting schemes (see e.g. \cite{sundness2006book,chen2017splitting, chen2019two}). These techniques have been largely preferred to fully implicit schemes such as the Bidomain monolithic solver proposed in \cite{murillo2004fully}. Fully implicit solvers can be computationally very expensive when the ionic model consists of very stiff and high-dimensional non-linear systems of ODEs 
	(e.g. \cite{luo1991model, ten2004model}). 
	On the other hand, operator splitting and decoupling techniques introduce additional errors which increase the time finite difference errors, see \cite[Ch. 7.2]{franzone2014mathematical} and the references therein. 
	As a possible compromise to balance computational effort and accuracy, we propose here a decoupling solution strategy based on a segregated implicit time discretization of our model, where at each time step we solve the ODEs system of the ionic model first, and then we solve and update the non-linear system arising from the discretized Bidomain equations.
	This strategy was studied previously in
	\cite{munteanu2009decoupled, munteanu2009scalable, scacchi2011multilevel}, where only overlapping one-level and Multilevel Schwarz preconditioners were considered.
	In this paper, we extend this solution strategy to two of the most efficient preconditioners currently available, the FETI-DP and BDDC dual-primal preconditioners with deluxe scaling.
	We recall that a practical advantage of these methods with respect to other domain decomposition preconditioners such as Multilevel Schwarz is that they can be easily extended to unstructured meshes, because they do not need the implementation of an inter-grid operator.
	
	FETI-DP methods were proposed by \cite{farhat2001feti} as an alternative to one-level or two-level FETI. They have been applied in several contexts, from three-dimensional elliptic problems with heterogeneous coefficients \cite{klawonn2002dual} to linear elasticity problems \cite{klawonn2006dual, rheinbach2006parallel}. In the biomechanics field, applications of FETI and FETI-DP have been extensively studied by \cite{augustin2014classical, brands2008modelling, klawonn2010highly, zampini2014dual}.
	BDDC were introduced by \cite{dohrmann2003preconditioner} as an alternative to FETI-DP for scalar elliptic problems and then analyzed by \cite{mandel2003convergence, mandel2005algebraic}. Among other applications, BDDC has been applied to the linearized (semi-implicit) Bidomain system in
	\cite{zampini2014dual, zampini2014inexact} and to cardiac mechanics in
	\cite{colli2018numerical, pavarino2015newton}. 
	FETI-DP and BDDC have been proven to be spectrally equivalent (\cite{li2006feti, mandel2005algebraic}). 
	The other main family of domain decomposition methods, based on the Overlapping Schwarz method, has also been applied to the Bidomain system, see \cite{PS2008,scacchi2008,charawi2017} for semi-implicit time discretizations and \cite{munteanu2009decoupled, munteanu2009scalable, scacchi2011multilevel} for segregated implicit time discretizations.
	
	Our main contribution consists in a novel theoretical estimate for the condition number of the preconditioned operator using the recent deluxe scaling introduced in \cite{dohrmann2016bddc}, for the solution of the non-linear system arising from a fully implicit discretization of the decoupled cardiac electrical model.
	We also present the results of extensive numerical tests confirming our optimality bound and the scalability of the proposed solver. Robustness of and computational equivalence between the proposed dual-primal algorithms are shown, thus encouraging further investigations with more complex realistic geometries and the development of monolithic solvers for electro-mechanical models.
	
	The rest of the paper is structured as follows. In Section \ref{cardiacmodel}, we briefly introduce the Bidomain model describing the electrical activity in the cardiac tissue. In Section \ref{numericalmethods}, we formulate the space discretization and 
	the time decoupling strategy. An overview of non-overlapping Domain Decomposition (DD) spaces and objects follows and an excursus of FETI-DP and BDDC preconditioners concludes Section \ref{ddprecond}. The novel convergence rate estimate is 
	proved in Section \ref{convrate}, while the results of several parallel numerical tests in three dimensions are presented in Section \ref{results}.

	\section{The cardiac Bidomain model}	\label{cardiacmodel}
	We consider here the non-linear parabolic reaction-diffusion system, describing the propagation of electrical signal through the cardiac tissue. 
	
	In the Bidomain system (\cite{franzone2014mathematical, pennacchio2005multiscale}), the cardiac tissue is represented as two interpenetrating domains, the intra- and extracellular domains, coexisting at every point of the cardiac tissue and connected by a distributed continuous cellular membrane, which (as the intra- and extracellular domains) fills the complete volume. 
	Cardiac cells are arranged in fibers set as laminar sheets running counterclockwise from the epicardium to the endocardium (see \cite{legrice1995laminar,franzone2014mathematical} for more details).
	In this way, at each point $\mathbf{x}$ of the cardiac domain $\Omega\subset \mathbb{R}^3$ it is possible to define an orthonormal triplet of vectors $\mathbf{a}_l(\mathbf{x})$, $\mathbf{a}_t(\mathbf{x})$ and $\mathbf{a}_n(\mathbf{x})$ parallel to the local fiber direction, tangent and orthogonal to the laminar sheets, respectively.
	
	We define the conductivity tensors of the two media as $D_{i,e}(\mathbf{x}) = \sum_{\bullet = \left\{l, t, n\right\} } \sigma_\bullet^{i,e} \mathbf{a}_\bullet (\mathbf{x})$, where $\sigma_\bullet^{i,e}$ are conductivity coefficients in the intra- and extracellular domain along the corresponding direction $\mathbf{a}_\bullet$, with $\bullet = {l, t, n}$. In our analysis, we assume here that these coefficients are constant in space. 
	By defining $u_{i,e}: \Omega \times (0, T) \rightarrow \mathbb{R}$ the intra- and extracellular electric potential, $v = u_i - u_e$ the transmembrane potential and by $w: \Omega \times (0, T) \rightarrow \mathbb{R}^{N_w}$, the gating and ionic concentration variables, the parabolic-parabolic formulation of the Bidomain system reads:
	\begin{eqnarray}\label{bido}
		\begin{cases}
			\chi C_m  \dvdt - \text{div} \left( D_i \cdot \nabla u_i \right) + \Ion (v,w ) = 0	  & \text{in } \Omega \times ( 0,T ),	\\
			-\chi C_m  \dvdt - \text{div} \left( D_e \cdot \nabla u_e \right) - \Ion (v,w) = - \Iapp^e	  &  \text{in } \Omega \times ( 0,T ),	\\
			\dwdt - R(v,w) = 0, 		&\text{in } \Omega \times ( 0,T ), \\
			v(x,t) = u_i(x,t) - u_e(x,t) 		&\text{in } \Omega \times ( 0,T ), \\
		\end{cases}
	\end{eqnarray}
	with initial values $ v(\mathbf{x},0) =   u_i(\mathbf{x},0) - u_e(\mathbf{x},0)$, $w(\mathbf{x},0) = w_0(\mathbf{x}) $, where $\chi$ is the ratio of membrane area per tissue volume and $C_m$ is the surface capacitance. 
	Assuming that the heart is immersed in a non-conductive medium, we require zero-flux boundary conditions $	\textbf{n}^T D_{i,e} \nabla u_{i,e}  = 0$ on $\pO \times ( 0,T )$ and compatibility condition $\int_{\Omega} \Iapp^e dx = 0$, where $\Iapp^{e}: \Omega \times (0, T) \rightarrow \mathbb{R}$ is the extracellular applied current. 
	Existence, uniqueness and regularity results for (\ref{bido}) have been extensively studied, see for example \cite{franzone2002degenerate}. 
	
	The ionic model describing the ionic currents flowing through the cell membrane is defined by the terms $\Ion$ and $R(v,w)$ in (\ref{bido}). In this work, we consider the very simple (yet macroscopically reliable) Rogers-McCulloch (RMC) ionic model (\cite{rogers1994collocation}), with only one gating variable and
	\begin{eqnarray}\nonumber
		\Ion (v,w)  = G \ v \left( 1 - \frac{v}{v_{th}}  \right) \left( 1 - \frac{v}{v_{p}}  \right) + \eta_1 v w,
		\qquad 	R(v,w) = \eta_2 \left( \frac{v}{v_p} - w \right),
	\end{eqnarray}
	where $G$, $v_{th}$, $v_p$, $\eta_1$ and $\eta_2$ are constant coefficients. \\

	\section{Numerical Methods} 		\label{numericalmethods}
	{\bf a) Space discretization.}
	The cardiac domain $\Omega$ is discretized by a structured quasi-uniform grid of hexahedral isoparametric $Q_1$ elements of maximal diameter $h$. Let $V_h \subset V$ be the associated finite element space, with the same basis functions $\left\{ \varphi_p \right\}_{p=1}^{N_h}$ for both variables $u_{i,e}$ and $w$ and let $A_{i,e}$ and $M$ be the stiffness and mass matrices with entries
	\begin{eqnarray} \nonumber
		\left\{ A_{i,e} \right\}_{nm} = \int_{\Omega} \left(  \nabla \varphi_n  \right)^T D_{i,e} \cdot \nabla \varphi_m dx , 	\qquad 		\left\{ M \right\}_{nm} = \int_{\Omega} \varphi_n \varphi_m dx.
	\end{eqnarray}
	For our purposes, the mass matrix is obtained with the usual mass-lumping technique.
	We denote by $\mathbf{u}_{i,e}$, $\mathbf{v}$, $\mathbf{w}$, $\mathbf{\Ion}$ and $\mathbf{\Iapp^{i,e}}$ the coefficient vectors from the discretization of $u_{i,e}$, $v$, $w$, $\Ion$ and $\Iapp^{i,e}$, respectively. 
	With these choices, we thus need to solve the semi-discrete Bidomain system
	\begin{eqnarray}
		\begin{cases}
			\chi C_m \mathcal{M} \frac{\partial}{\partial t} 
			\begin{bmatrix}
				\mathbf{u}_i \\ \mathbf{u}_e
			\end{bmatrix}
			+ \mathcal{A} 
			\begin{bmatrix}
				\mathbf{u}_i \\ \mathbf{u}_e
			\end{bmatrix}
			+ 
			\begin{bmatrix}
				M \ \mathbf{\Ion} (\mathbf{v}, \mathbf{w}) \\ -M \ \mathbf{\Ion} (\mathbf{v}, \mathbf{w}) 
			\end{bmatrix}
			= 
			\begin{bmatrix}
				\mathbf{0} \\ - M \ \mathbf{\Iapp^e}
			\end{bmatrix},
			\\
			\frac{\partial \mathbf{w}}{\partial t} = R \left( \mathbf{v}, \mathbf{w} \right),
		\end{cases} 
	\end{eqnarray}
	where $\mathcal{A}$ and $\mathcal{M}$ are the stiffness and mass block-matrices
	$	\displaystyle\mathcal{A} =
	\begin{bmatrix}
		A_i 	& 0 	\\
		0 		& A_e
	\end{bmatrix},
	\ \ \ \
	\mathcal{M} = 
	\begin{bmatrix}
		M 		& -M 	\\
		-M 		& M
	\end{bmatrix}.
	$\\
	
	{\bf b) Decoupled implicit time discretization.}
	Fully implicit discretizations in time of the Bidomain system coupled with ionic models while using physiological coefficients, lead to the solution of non-linear problems at each time step, which can be very expensive from a computational point of view: indeed, realistic ionic models are very complex and can present up to fifty non-linear ODEs 
	(see e.g. \cite{luo1991model, ten2004model}). 
	Few attempts in this direction have been done by using simple ionic models (e.g. \cite{murillo2004fully}).
	In the literature, common alternatives consider implicit-explicit (IMEX) schemes and/or operator splitting, where the diffusion terms are treated separately from the reaction (e.g. \cite{chen2017splitting, chen2019two, colli2018numerical, zampini2014dual}); see also \cite[Ch. 7.2]{franzone2014mathematical} and the references therein.
	
	Instead, we choose here to decouple the gating variable $w$ from the intra- and extracellular potentials $u_i$ and $u_e$ as in \cite{munteanu2009decoupled}.  
	At each time step, this decoupled Backward Euler strategy consists of two more sub-steps.
	\begin{itemize}
		\item {\bf Step 1}: update gating and ionic variables. 
		Given $\mathbf{u}_{i,e}^n$ (hence $\mathbf{v}^n$) at the previous time step $t_n$, compute $\mathbf{w}^{n+1}$ by solving the membrane model
		\begin{eqnarray}\nonumber
			\mathbf{w}^{n+1} - \tau \mathbf{R} (\mathbf{v}^n, \mathbf{w}^{n+1}) = \mathbf{w}^n,
		\end{eqnarray}
		where $\tau = t_{n+1} - t_n$ is the current time step.
		
		\item {\bf Step 2}: solve the Bidomain system.
		Given $\mathbf{u}_{i,e}^n$ at the previous time step and given $\mathbf{w}^{n+1}$, calculate $\mathbf{u}^{n+1} = (\mathbf{u}_i^{n+1}, \mathbf{u}_e^{n+1})$ by solving the non-linear equation $\mathbf{F} (\mathbf{u}^{n+1}) = \mathbf{G}$ derived from the Backward Euler scheme applied to the Bidomain equations, where
		\begin{eqnarray}\nonumber
			\!\!\!\!\!\mathbf{F}(\mathbf{u}^{n+1}) = \left( \chi C_m \mathcal{M} + \tau \mathcal{A} \right) 
			\begin{bmatrix}
				\mathbf{u}_i^{n+1} \\
				\mathbf{u}_e^{n+1}
			\end{bmatrix}
			+ \tau
			\begin{bmatrix}
				M \mathbf{\Ion}(\mathbf{v}^{n+1}, \mathbf{w}^{n+1}) \\
				-M \mathbf{\Ion}(\mathbf{v}^{n+1}, \mathbf{w}^{n+1}) 
			\end{bmatrix},
			\ \ \ \ \ \ 		\mathbf{G} =  \chi C_m \mathcal{M} 
			\begin{bmatrix}
				\mathbf{u}_i^n \\
				\mathbf{u}_e^n
			\end{bmatrix}
			+ \tau
			\begin{bmatrix}
				\mathbf{0} \\
				-M \mathbf{\Iapp^e} 
			\end{bmatrix}.
		\end{eqnarray}
	\end{itemize}	
	
	The strategy proposed here for the solution of the non-linear system in Step 2 consists in a Newton-Krylov approach,  as in \cite{munteanu2009decoupled, munteanu2009scalable, scacchi2011multilevel}, where the decomposition of the problem is made after the linearization:  a Newton scheme is applied as outer iteration and the Jacobian linear system arising at each Newton step is solved by a Krylov method
	with a dual-primal preconditioner.\\
	
	\indent {\bf c) Newton scheme and properties of the Jacobian bilinear form.}	
	The outer Newton iteration can be summarized as
	\begin{itemize}
		\item[-] choose a starting value $\mathbf{u}^0 = (\mathbf{u}^0_i, \mathbf{u}^0_e)$;
		\item[-] for $k \geq 0$ solve the Jacobian linear system
		\begin{eqnarray}\label{Jacobiansystem}
			\mathbf{JF}^k \mathbf{s}^k = - \mathbf{F} (\mathbf{u}^k)
		\end{eqnarray}
		until a Newton stopping criterion is satisfied, where $\mathbf{s}^k = (\mathbf{s}^k_i, \mathbf{s}^k_e )$ is the Newton correction at step $k$ and $\mathbf{JF}^k$ is the Jacobian of $\mathbf{F}$ computed in $\mathbf{u}^k$.
		\item[-] update: $\mathbf{u}^{k+1} = \mathbf{u}^k + \mathbf{s}^k$.
	\end{itemize}
	For our theoretical purposes, we need to associate to problem (\ref{Jacobiansystem}) the bilinear form 
	\begin{equation*}
		a (\skk, \phi) := \chi C_m \innerprod{\sikk\!-\!\sekk}
		{\varphi_i\!-\!\varphi_e}+\tau \bilformi{\sikk}{\varphi_i}+
		\tau \bilforme{\sekk}{\varphi_e} 
		+\tau \innerprod{\sumdidv \left(\silkk\!-\!\selkk \right)\psi_l}{\varphi_i\!-\!\varphi_e}\!.
	\end{equation*}
	for all $ \skk = \left( \sikk, \sekk \right) \in V_h \times V_h$ and $\phi = \left( \varphi_i, \varphi_e \right) \in V_h \times V_h$, being $\psi_l$ the $l$-th nodal basis function, where $a_{i,e}(\cdot,\cdot)$ are the bilinear forms associated with the diffusion terms and $\innerprod{\cdot}{\cdot}$ denotes the usual $L^2$-inner product.
	As in \cite{munteanu2009decoupled}, it is possible to show that this bilinear form associated to the Jacobian linear system is continuous and coercive with respect to the following norm defined $\forall u = (u_i, u_e) \in V_h \times V_h$
	\begin{equation*}
		\normtau{u}^2 := (1 + \tau) \norm{u_i - u_e}{\Omega}^2 + \tau \bilformi{u_i}{u_i} + \tau \bilforme{u_e}{u_e}.
	\end{equation*}
	\begin{lemma}
		Assume that 
		\begin{equation*}
			\chi C_m + \tau \pIon \geq c > 0, \qquad c \in \mathbb{R}^+,
		\end{equation*}
		holds $\forall l = 1, \dots, N_h$ and for all $k$.
		Then the bilinear form $a(\cdot,\cdot)$ is continuous and coercive with respect to the norm $\normtau{\cdot}$.
	\end{lemma}
	
	\begin{remark}
		\emph{We do observe that the above hypothesis of non-negativity is always satisfied for any time step $\tau \leq 0.37$ ms, using the Rogers-McCulloch ionic model. Indeed, numerical computations of $\chi C_m + \tau \frac{\partial I_\text{ion}}{\partial v}$ validate this assumption (see Fig. (\ref{hypplot})).} 
	\end{remark}
	
	\begin{figure}[!h]
		\centering
		\includegraphics[scale=.6]{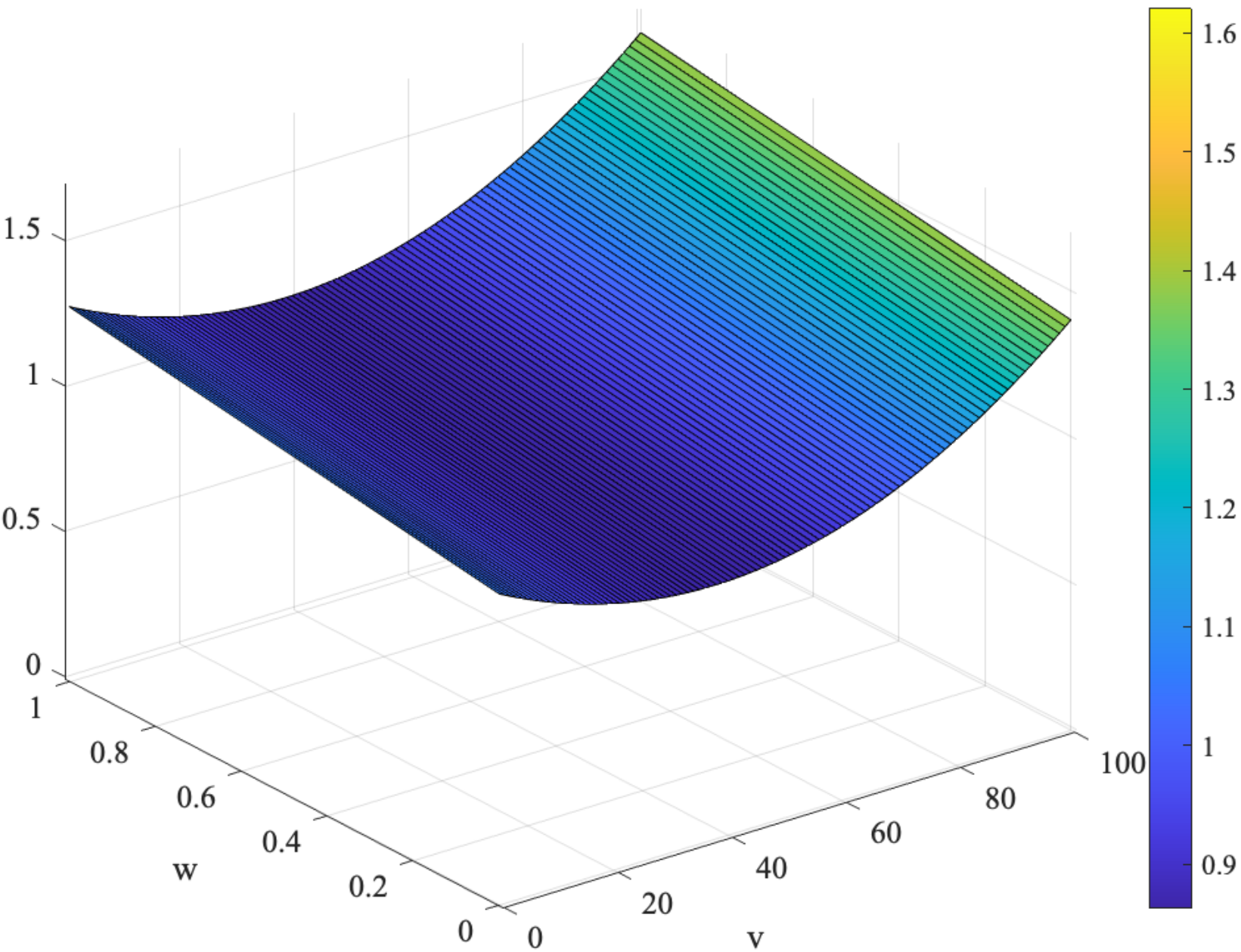}
		\caption{Surface plot of $\chi C_m + \tau \frac{\partial I_\text{ion}}{\partial v}$, with $C_m = 1 \frac{mF}{cm^3}$, $\chi = 1$ and $\tau=0.05$ ms, which are values usually employed in numerical experiments.}
		\label{hypplot}
	\end{figure}
	
	As an immediate consequence of the continuity and coercivity of the bilinear form $a (\cdot,\cdot)$, the following Lemma holds.
	We drop the index $k$ from now on, unless an explicit ambiguity occurs.
	\begin{lemma}\label{ellipbound}
		Assuming that the conductivity coefficients are constant in space, the bilinear form $a (\cdot, \cdot)$ satisfies the bounds
		\begin{equation*}
			\begin{aligned}
				a (s,s) \leq \left( \chi C_m + \tau K_M \right) \norm{s_i - s_e}{\Omega}^2 + \tau \sigma^i_M \seminormHone{s_i}{\Omega}^2 + \tau \sigma^e_M \seminormHone{s_e}{\Omega}^2, \\
				a (s, s) \geq \left( \chi C_m + \tau K_m \right) \norm{s_i - s_e}{\Omega}^2 + \tau \sigma^i_m \seminormHone{s_i}{\Omega}^2 + \tau \sigma^e_m \seminormHone{s_e}{\Omega}^2,
			\end{aligned}
		\end{equation*}
		where 
		\begin{equation*}
			\sigma^{i,e}_M = \max_{\bullet = \left\{ l, t, n \right\}} \sigma^{i,e}_\bullet, 		\qquad \qquad  \sigma^{i,e}_m = \min_{\bullet = \left\{ l, t, n \right\}} \sigma^{i,e}_\bullet,
		\end{equation*}
		and $K_M$, $K_m$ independent from the subdomain diameter $H$ and the mesh size $h$.
	\end{lemma}
	
	\begin{remark}
		\emph{This result is extensible to the case of conductivity coefficients almost constant over each subdomain.}
	\end{remark}

	\section{Dual-primal preconditioners for Newton-Krylov solvers} \label{ddprecond}
	
	\subsection{ Non-overlapping dual-primal spaces} 
	Let $\Omega_j$, $j = 1, \dots, N$, be a decomposition of the cardiac domain $\Omega$, into non-overlapping subdomains (or substructures).  This decomposition forms a partition of $\Omega$, such that $\overline{\Omega} = \cup_{j=1}^{N} \overline{\Omega}_j$, $ \Omega_j \cap \Omega_k = \emptyset $ if $ j \neq k$ and the intersection of the boundaries $ \pO_j \cap \pO_k $ is either empty, a vertex, an edge or a face. 
	Each subdomain is a union of shape-regular conforming finite elements.
	The interface $\Gamma$ is the set of points that belong to at least two subdomains,
	\begin{equation*}
		\Gamma := \bigcup\limits_{j \neq k} \pO_j \cap \pO_k, 	\qquad \qquad 	\Gammaj = \pO_j \cap \Gamma.
	\end{equation*} 
	We assume that subdomains are shape regular and have a typical diameter of size $H_j$ whereas the finite elements are of diameter $h$; we denote by $H = \max_j H_j$. 
	Let $W^{(j)} = V_h(\Omega_j) \times V_h(\Omega_j)$ be the associated local finite element spaces. We partition $W^{(j)}$ into the interior part $W_I^{(j)}$ and the finite element trace space $W_\Gamma^{(j)}$, such that
	Note that we consider variables on the Neumann boundaries $\pO_N$ as interior to a subdomain. We introduce the product spaces 
	\begin{equation*}
		W = \prod_{j=1}^{N} W^{(j)}, \qquad \qquad W_\Gamma :=  \prod_{j=1}^{N} W_\Gamma^{(j)}.
	\end{equation*}
	Therefore we define $\widehat{W} \subset W$ as the subspace of functions of $W$, which are continuous in all interface variables between subdomains and similarly we denote by $\widehat{W}_\Gamma \subset W_\Gamma$, the subspace formed by the continuous elements of $W_\Gamma$. 
	In dual-primal methods, we iterate in the space $W$ while requiring continuity constraints (also called \textit{primal constraints}) to hold throughout the iterations. Primal constraints guarantee that each subdomain problem is invertible and that a good convergence bound can be obtained.
	We denote by $\widetilde{W}$ the space of finite element functions in $W$, which are continuous in all primal variables; clearly we have $\widehat{W} \subset \widetilde{W} \subset W$ and likewise $\widehat{W}_\Gamma \subset \widetilde{W}_\Gamma \subset W_\Gamma$.  
	Let $W_\Pi^{(j)} \subset W_\Gamma^{(j)}$ be the primal subspace of functions which are continuous across the interface and that will be subassembled between the subdomains sharing $\Gamma^{(j)}$. 
	Additionally, let \mbox{$W_\Delta^{(j)} \subset 	W_\Gamma^{(j)}$} contain the finite element functions (called dual) which can be 
	discontinuous across the interface and vanish at the primal degrees of freedom. 
	
	We then introduce the product subspaces 
	$W_\Pi = \prod_{j=1}^{N} W_\Pi^{(j)}$ and $W_\Delta = \prod_{j=1}^{N} W_\Delta^{(j)}$, from which $W_\Gamma = W_\Pi \oplus W_\Delta$.
	Using this notation, we can decompose $\widetilde{W}_\Gamma$ into a primal subspace $\widehat{W}_\Pi$ which has continuous elements only and a dual subspace $W_\Delta$ which contains finite element functions which are not continuous, i.e. we have $\widetilde{W}_\Gamma = \widehat{W}_\Pi \oplus W_\Delta$. 
	In this work, we will denote with subscripts $I$, $\Delta$ and $\Pi$ the interior, dual and primal variables, respectively. 
	
	In the non-overlapping framework, the global system matrix (see in this application Eq. (\ref{Jacobiansystem})) is never formed explicitly, but a local version with the same structure is assembled on each subdomain, by restricting the integration set from $\Omega$ to $\Omega_j$ by defining the local bilinear forms 
	\begin{equation*}
		\begin{split}
			a^{(j)} (s, \phi) = \chi C_m \innerprod{s_i - s_e}{\varphi_i - \varphi_e}_{|\Omega_j} + \tau \bilformiloc{s_i}{\varphi_i} + \tau \bilformeloc{s_e}{\varphi_e}\\
			\qquad \qquad \qquad + \tau \innerprod{\sumdidve \left( s_{i,l} - s_{e,l} \right) \varphi_l}{\varphi_i - \varphi_e}_{|\Omega_j},
		\end{split}
	\end{equation*}
	where $\innerprod{\cdot\ }{\cdot}_{|\Omega_j}$ denotes the restriction of the $L^2$-inner product on the $j$-th subdomain. 
	These definitions are admissible here, as the proposed theory allows constant non-negative distribution of the conductivity coefficients among all subdomains, with large jumps aligned to the interfaces.
	
	The reordering of the degrees of freedom leads to a reordered system matrix: assuming that the system (\ref{Jacobiansystem}) can be written as $\mathcal{K} u = f$, then we will have
	\begin{equation*}
		\mathcal{K}^{(j)} = 
		\begin{bmatrix}
			K_{II}^{(j)} 					& K_{I\Gamma}^{(j)} \\
			K_{I \Gamma}^{(j) T}  	& K_{\Gamma \Gamma}^{(j)}
		\end{bmatrix},
		\qquad
		\mathcal{K} = 
		\begin{bmatrix}
			K_{II} 					& K_{I\Gamma} \\
			K_{I \Gamma}^T 	& K_{\Gamma \Gamma}
		\end{bmatrix}.
	\end{equation*}
	As in classical iterative substructuring, we apply the so called {\it static condensation}, which consists in eliminating the interior degrees of freedom, thus reducing the problem to one on the interface $\Gamma = \cup_{j=1}^N \partial \Omega_j \backslash \partial \Omega$. The local Schur complement systems are
	\begin{equation*}
		S_{\Gamma}^{(j)} = K_{\Gamma \Gamma}^{(j)} - K_{I \Gamma}^{(j) T} K_{II}^{(j) -1}  	 K_{I\Gamma}^{(j)}.
	\end{equation*}
	By defining the unassembled Schur complement matrix $S_\Gamma = \text{diag } \left[ S_\Gamma^{(1)}, \dots, S_\Gamma^{(N)} \right]$, we obtain the global Schur complement matrix $\widehat{S}_\Gamma = R_\Gamma^T S_\Gamma R_\Gamma^T$, where $R_\Gamma$ is the direct sum of local restriction operators $R_\Gamma^{(j)}$ returning the local interface components. 
	Thus, instead of solving system (\ref{Jacobiansystem}), we solve the Schur complement system 
	\begin{equation}\label{schursys}
		\widehat{S}_\Gamma u_\Gamma = \widehat{f}_\Gamma, 
	\end{equation}
	where $\widehat{f}_\Gamma$ is retrieved from the right-hand-side of \ref{Jacobiansystem}.
	Once this problem is solved, the solution $u_\Gamma$ on the interface is used to recover the solution on the internal degrees of freedom
	$u_I = K_{II}^{-1} \left( f_I  - K_{I \Gamma} u_\Gamma \right)$.
	The Schur complement matrix $\widehat{S}_\Gamma$ of the Jacobian Bidomain system (\ref{Jacobiansystem}) is symmetric, positive semidefinite, hence it is possible to apply the Preconditioned Conjugate Gradient (PCG) method.
	
	We define the Jacobian Bidomain local discrete harmonic extension operators as follows:
	\begin{equation*}
		\mathcal{H}_j : W_\Gamma^{(j)} \longrightarrow W^{(j)}, \qquad \mathcal{H}_j w_\Gamma^{(j)} = 
		\begin{dcases*}
			- K_{II}^{(j) -1}  	 K_{I\Gamma}^{(j)}w_\Gamma^{(j)}		\qquad		\text{on } W^{(j)}_I	\\
			w_\Gamma^{(j)}													\qquad	\qquad \qquad \qquad		\text{on } W^{(j)}_\Gamma
		\end{dcases*}.
	\end{equation*}
	We note that the local discrete harmonic extension of a constant vector is the vector itself. 
	From now on, we will use the component-wise notation 
	\begin{equation*}
		\mathcal{H}_j w_\Gamma^{(j)} = \left( \mathcal{H}_j^i w_\Gamma^{(j)}, \mathcal{H}_j^e w_\Gamma^{(j)} \right),
	\end{equation*}
	where the superscripts $i, e$ denote the usual intra- and extracellular components.
	As for standard elliptic problems (see \cite{toselli2006domain}), the Schur bilinear form can be defined through the action of the Schur matrix and the Jacobian bilinear form 
	\begin{equation*}
		a^{(j)} (\mathcal{H}_j u_\Gamma^{(j)} , \mathcal{H}_j v_\Gamma^{(j)} ) =  v_\Gamma^{(j)} S_\Gamma^{(j)} u_\Gamma^{(j)} = s^{(j)} (u_\Gamma^{(j)},v_\Gamma^{(j)}).
	\end{equation*}
	From the definition of $S_\Gamma^{(j)}$, it follows immediately that the bilinear form $s^{(j)}(\cdot, \cdot)$ is symmetric and coercive. Thanks to Lemma \ref{ellipbound}, it is possible to bound the energies related to the local Schur complements, 
	\begin{equation}\label{schurminform}
		s^{(j)} (u_\Gamma^{(j)},u_\Gamma^{(j)}) = \min_{v^{(j)}_{| \pO_j \cap \Gamma} =  u_\Gamma^{(j)} } a^{(j)} (\mathcal{H}_j u_\Gamma^{(j)} , \mathcal{H}_j u_\Gamma^{(j)} ),
	\end{equation}
	which allows us to work with discrete harmonic extensions
	instead of functions defined only on $\Gamma$. 
	
	\subsection{Restriction operators and scaling}
	We define the restriction operators 
	\begin{eqnarray}
		R_\Delta^{(j)}: W_\Delta \rightarrow W_\Delta^{(j)},  		&\qquad R_{\Gamma \Delta}: W_\Gamma \rightarrow W_\Delta, 	\nonumber \\
		R_\Pi^{(j)}: \widehat{W}_\Pi \rightarrow W_\Pi^{(j)},  	 	&\qquad R_{\Gamma \Pi}: W_\Gamma \rightarrow \widehat{W}_\Pi, 	\nonumber
	\end{eqnarray}
	and the direct sums $R_\Delta = \oplus R_\Delta^{(j)}$, $R_\Pi = \oplus R_\Pi^{(j)} $ and $\widetilde{R}_\Gamma = R_{\Gamma \Pi} \oplus R_{\Gamma \Delta}$, which maps $W_\Gamma $ into $\widetilde{W}_\Gamma$. 
	We also need a proper scaling of the dual variables.\\
	
	{\bf $\rho$-scaling.} Originally proposed for Neumann-Neumann methods, the $\rho$-scaling is defined  for the Bidomain model at each node $x \in \Gammaj$ as 
	\begin{equation}\label{pseudoinv}
		\diej (x) = \dfrac{\sigma_M^{{i,e}^{(j)}}}{\sumNx{\sigma_M^{{i,e}^{(k)}}}}, 		\qquad 			\sigma_M^{{i,e}^{(j)}} = \max_{\bullet = \left\{ l, t, n \right\}} \sigma^{{i,e}^{(j)}}_\bullet,
	\end{equation}
	where $\mN_x$ is the set of indices of all subdomains with $x$ in the closure of the subdomain. If $x$ is in the interior of a subdomain, then $\mN_x$ contains only the index of that subdomain. Moreover, $\mN_x$  induces the definition of an equivalence relation that allows the classification of interface degrees of freedom into faces, edges and vertices equivalence classes. \\
	
	{\bf Deluxe scaling.}  Recently introduced in \cite{dohrmann2016bddc} and studied in \cite{da2014isogeometric}, the deluxe scaling computes the average $\bar{w} = \Ed w$ for each face $\mathcal{F}$ or edge $\mathcal{E}$ equivalence class.
	Suppose that $\mathcal{F}$ is shared by subdomains $\Omega_j$ and $\Omega_k$. Let $\SjF$ and $\SkF$ be the principal minors obtained from $\Sj_\Gamma$ and $\Sk_\Gamma$ by removing all the contributions that are not related to the degrees of freedom of the face $\mathcal{F}$.  
	Let $\uFj = \RF u_j$ be the restriction of $u_j$ to the face $\mathcal{F}$ through the restriction operator $\RF$. The deluxe average across $\mathcal{F}$ is then defined as 
	\begin{eqnarray}\nonumber
		\umeanF = \left( \SjF + \SkF \right)^{-1} \left( \SjF \uFj + \SkF \uFk \right).
	\end{eqnarray}
	The action of $ ( \SjF + \SkF )^{-1} $ can be computed by solving a Dirichlet problem over the two subdomains involved, by extending to zero the right-hand side entries that correspond with the interior nodes. 
	
	If we consider an edge $\mathcal{E}$ instead, 
	the deluxe average across $\mathcal{E}$ is defined in a similar manner. 
	Suppose for simplicity that $\mathcal{E}$ is shared by only three subdomains with indices $j_1$, $j_2$ and $j_3$; the extension to more than three subdomains is immediate. 
	Let $\uEj = \RE u_j$ be the restriction of $u_j$ to the edge $\mathcal{E}$ through the restriction operator $\RE$ and define $\SjallE = \SjunoE + \SjdueE + \SjtreE$; the deluxe average across an edge $\mathcal{E}$ is given by
	\begin{eqnarray}\nonumber
		\umeanE = \left( \SjallE  \right)^{-1} \left( \SjunoE \uEjuno + \SjdueE \uEjdue + \SjtreE \uEjtre \right).
	\end{eqnarray}
	
	The relevant equivalence classes, involving the substructure $\Omega_j$, will contribute to the values of $\bar{u}$. These contributions will belong to $\widehat{W}_\Gamma$, after being extended by zero to $\Gamma \backslash \mathcal{F}$ or $\Gamma \backslash \mathcal{E}$; the sum of all contributions will result in $R^T_\ast \bar{u}_\ast$. We then add the contributions from the different equivalence classes to obtain
	\begin{eqnarray}\nonumber
		\bar{u} = \Ed u = u_\Pi + \sum_{\ast = \{ \mathcal{F}, \mathcal{E} \} } R^T_\ast \bar{u}_\ast,
	\end{eqnarray}
	where $\Ed$ is a projection. Its complementary projection is given by
	\begin{eqnarray}\label{PdDeluxe}
		\Pd u := (I - \Ed) u = u_\Delta - \sum_{\ast = \{ \mathcal{F}, \mathcal{E} \} } R^T_\ast \bar{u}_\ast.
	\end{eqnarray}
	For each subdomain $\Omega_j$ we define the scaling matrix
	\begin{equation}\label{deluxescaling}
		D^{(j)} =
		\begin{bmatrix}
			D^{(j)}_{\ast_{k_1}} 	&		&	\\
			& \ddots	&	\\
			&		&D^{(j)}_{\ast_{k_j}}
		\end{bmatrix},
		\qquad
		\ast = \left\{ \mathcal{F}, \mathcal{E} \right\}
	\end{equation}
	with $k_1, \dots, k_j \in \varXi_j^\ast$ set containing the indices of the subdomains that share the face $\mathcal{F}$ or the edge $\mathcal{E}$ and where the diagonal blocks are given by $D^{(j)}_\mathcal{F} = ( \SjF + \SkF )^{-1} \SjF $ or $D^{(j)}_\mathcal{E}=  (\SjunoE + \SjdueE + \SjtreE)^{-1} \SjunoE$.
	
	We can now define the scaled local restriction operators
	\begin{equation*}
		R_{D, \Gamma}^{(j)} = D^{(j)} R_\Gamma^{(j)}, 	\qquad 	\qquad R_{D, \Delta}^{(j)} = R_{\Gamma \Delta}^{(j)} R_{D, \Gamma}^{(j)} ,
	\end{equation*}
	$R_{D, \Delta}$ as direct sum of $R_{D, \Delta}^{(j)}$ and the global scaled operator 	$\widetilde{R}_{D, \Gamma} = R_{\Gamma \Pi} \oplus R_{D, \Delta} R_{\Gamma \Delta}$. \\
	
	
	\subsection{FETI-DP preconditioner}
	The Finite Element Tearing and Interconnecting Dual-Primal preconditioner was first proposed in \cite{farhat2001feti} as an alternative to one-level and two-level FETI. This class of methods is based on the transposition from the Schur problem (\ref{schursys}) on $\widehat{W}_\Gamma$ to a minimization problem on $\widetilde{W}_\Gamma$, with continuity constraints on the dual degrees of freedom: 
	find $w_\Gamma \in \widetilde{W}_\Gamma$ which minimizes
	\begin{equation*}
		\begin{dcases*}
			\frac{1}{2} w_\Gamma^T \widetilde{S}_\Gamma w_\Gamma - w_\Gamma^T \tilde{f}_\Gamma 	\\
			B w_\Gamma = 0
		\end{dcases*},
	\end{equation*}
	where $\tilde{f}_\Gamma = \widetilde{R}_\Gamma^T f_\Gamma$ is given by partially subassembling the Schur complement right-hand side on primal nodes and $B$ is the jump operator with entries $0, \pm 1$. The second part of the system $B w_\Gamma = 0$ holds if and only if $ w_\Gamma \in \widehat{W}$, which means that the columns of $B$ related to primal degrees of freedom are null.
	By introducing a set of Lagrange multipliers $\lambda \in \Lambda = range (B)$, it is possible to formulate the minimization problem as a saddle point system
	\begin{equation*}
		\begin{bmatrix}
			\widetilde{S}_\Gamma 	& B^T		\\
			B									& 0 
		\end{bmatrix}
		\begin{bmatrix}
			w_\Gamma	\\				\lambda	
		\end{bmatrix}
		=
		\begin{bmatrix}
			\tilde{f}_\Gamma 	\\		0
		\end{bmatrix}.
	\end{equation*}
	As $\widetilde{S}_\Gamma$ is invertible on $range (B)$, the degrees of freedom in $\widetilde{W}_\Gamma$ can be eliminated by a block-Cholesky factorization, reducing the above system to a problem only in the Lagrange multipliers unknowns
	\begin{equation}\label{lagrangesystem}
		F \lambda = d, 			\qquad 			\text{where } \quad F = B \widetilde{S}_\Gamma^{-1} B^T, 		\qquad 			d = B \widetilde{S}_\Gamma^{-1}  \tilde{f}_\Gamma.
	\end{equation}
	After the solution $\lambda$ is found, we can retrieve the solution on $\widetilde{W}_\Gamma$ as $w_\Gamma = \widetilde{S}_\Gamma^{-1} \left( \tilde{f}_\Gamma - B^T \lambda \right)$.
	The FETI-DP system (\ref{lagrangesystem}), in our application is symmetric, thus PCG works well. In order to ensure fast convergence, a quasi-optimal preconditioner is given by
	\begin{equation*}
		M^{-1}_\text{FETI-DP} = B_D \widetilde{S}_\Gamma B_D^T,
	\end{equation*}
	where $B_D$ is the scaled jump operator, obtained by applying $D^{(j)} : \Lambda \rightarrow \Lambda$ scaling matrices that act on the space of the Lagrange multipliers, which are given by (\ref{deluxescaling}) if the deluxe scaling is used or by the pseudoinverses (\ref{pseudoinv}) if the standard $\rho$-scaling is used.
	
	\subsection{BDDC preconditioner}
	BDDC is a two-level preconditioner for the Schur complement system $\widehat{S}_\Gamma u_\Gamma = \widehat{f}_\Gamma$.
	If we partition the degrees of freedom of the interface $\Gamma$ into those internal ($I$) and those dual ($\Delta$), the matrix $\mathcal{K}^{(j)}$ from the problem $\mathcal{K} u = f$ can be written as
	\begin{equation*}
		\mathcal{K}^{(j)} = 
		\begin{bmatrix}
			K_{II}^{(j)} 					& K_{I\Gamma}^{(j)} \\
			K_{I \Gamma}^{(j) T}  	& K_{\Gamma \Gamma}^{(j)}
		\end{bmatrix}
		=
		\begin{bmatrix}
			K_{II}^{(j)} 					& K_{I \Delta}^{(j)} 			   & K_{I \Pi}^{(j)} \\
			K_{I \Delta}^{(j) T}  		& K_{\Delta \Delta}^{(j)} 		& K_{\Delta \Pi}^{(j)}	\\
			K_{I \Pi}^{(j) T}			  & K_{\Delta \Pi}^{(j) T}			& K_{\Pi \Pi}^{(j)}
		\end{bmatrix}.
	\end{equation*} 
	It is possible to define the BDDC preconditioner using the restriction operators as 
	\begin{equation*}
		M^{-1}_\text{BDDC} = \widetilde{R}_{D, \Gamma}^T \widetilde{S}_\Gamma^{-1}  \widetilde{R}_{D, \Gamma}, 	\qquad 		\widetilde{S}_\Gamma = \widetilde{R}_\Gamma S_\Gamma \widetilde{R}_\Gamma^T,
	\end{equation*}
	where the action of the inverse of $\widetilde{S}_\Gamma$ can be evaluated with a block-Cholesky elimination procedure
	\begin{equation*}
		\widetilde{S}_\Gamma^{-1} = \widetilde{R}_{\Gamma \Delta }^T \left( \sum_{j=1}^{N} 
		\begin{bmatrix}
			0 		&R_\Delta^{(j) T}
		\end{bmatrix}
		\begin{bmatrix}
			K_{II}^{(j)} 					& K_{I \Delta}^{(j)} 			  \\
			K_{I \Delta}^{(j) T}  		& K_{\Delta \Delta}^{(j)}
		\end{bmatrix}^{-1}
		\begin{bmatrix}
			0 	\\		R_\Delta^{(j)}
		\end{bmatrix}
		\right) \widetilde{R}_{\Gamma \Delta } + \varPhi S_{\Pi \Pi}^{-1} \varPhi .
	\end{equation*}
	In this way, the first term is the sum of local solvers on each substructure $\Omega_j$, while the latter is a coarse solver for the primal variables where
	\begin{equation*}
		\varPhi = R_{\Gamma \Pi}^T - R_{\Gamma \Delta}^T \sum_{j=1}^N 
		\begin{bmatrix}
			0 		&R_\Delta^{(j) T}
		\end{bmatrix}
		\begin{bmatrix}
			K_{II}^{(j)} 					& K_{I \Delta}^{(j)} 			  \\
			K_{I \Delta}^{(j) T}  		& K_{\Delta \Delta}^{(j)}
		\end{bmatrix}^{-1}
		\begin{bmatrix}
			K_{I \Pi}^{(j)} 	\\		R_{\Delta \Pi}^{(j)}
		\end{bmatrix}
		R_\Pi^{(j)} ,
	\end{equation*}
	\begin{equation*}
		S_{\Pi \Pi } = \sum_{j=1}^N  R_\Pi^{(j) T} 
		\left( K_{\Pi \Pi}^{(j)} - 
		\begin{bmatrix}
			K_{I \Pi}^{(j) T}		&K_{\Delta \Pi}^{(j) T}   
		\end{bmatrix}
		\begin{bmatrix}
			K_{II}^{(j)} 					& K_{I \Delta}^{(j)} 			  \\
			K_{I \Delta}^{(j) T} 		& K_{\Delta \Delta}^{(j)}
		\end{bmatrix}^{-1}
		\begin{bmatrix}
			K_{I \Pi}^{(j)} 	\\		R_{\Delta \Pi}^{(j)}
		\end{bmatrix}
		\right) R_\Pi^{(j)},
	\end{equation*}
	are the matrix which maps the primal degrees of freedom to the interface variables and the primal problem respectively.

	
	\section{Convergence rate estimate} \label{convrate}
	It has been proven that FETI-DP and BDDC methods are spectrally equivalent \cite{li2006feti}. In this perspective, we are able to prove a convergence rate estimate for the preconditioned operator, which holds for both methods when the same coarse space is chosen.
	We observe that, as in most of the convergence bounds for FETI-DP and BDDC operator, also in this application the condition number is independent of the number of subdomains.
	We first recall some useful technical results that will be employed in the proof of the convergence rate estimate. These results can be found in the appendix A of \cite{toselli2006domain}. 
	\begin{theorem}[Trace theorem]\label{tracethm}
		Let $\Omega_j$ be a polyhedral domain and define the discrete harmonic extension of the Laplacian operator $\harmext{j}{u_\Gamma}$ on $\Omega_j$ as
		\begin{equation*}
			u = \harmext{j}{u_\Gamma} \Leftrightarrow 
			\begin{dcases*}
				- \Delta u = 0	\qquad  \text{in } \Omega_j \\
				u = u_\Gamma    \quad \qquad \text{on } \Gammaj \\
				u = 0 		\qquad \qquad \text{on } \pO_j \backslash \Gammaj.
			\end{dcases*}.
		\end{equation*} 
		Then,
		\begin{equation*}
			\seminormHhalf{u}{\Gammaj}^2 \sim \seminormHone{\harmext{j}{u_\Gamma}}{\Omega_j}^2,
		\end{equation*}
	\end{theorem}
	\begin{proposition}[Poincar\'e-Friedrichs inequality]\label{frieddisug}
		Let $\Omega$ be Lipschitz continuous with diameter $H$. Then, there exists a constant $C_f$, that depends only on the shape of $\Omega$ but not on its size, such that 
		\begin{equation*}
			\norm{u}{\Omega}^2 \leq C_f H^2 |u|^2_{H^1(\Omega)},
		\end{equation*}
		for all $u \in H^1(\Omega)$ with vanishing mean value on $\Omega$. 
	\end{proposition}
	
	We will write with $A \lesssim B$ whenever $A \leq cB$ with $c$ constant independent from the diameter $H$, the mesh size $h$, the time step $\tau$ and the conductivity coefficients; similarly, we will write $A \sim B$ whenever $A \lesssim B$ and $B \lesssim A$. 
	The main result of the paper is the following optimality bound.
	\begin{theorem} 
		If the deluxe scaling is used, the condition number of the FETI-DP and BDDC preconditioned operators satisfy
		\begin{eqnarray}
			cond \ (P^{-1} Q) \leq \max_{\substack{ k = 1,\dots,N \\ \star = i,e}}  \dfrac{\tau \sigma_M^{\star (k)} + H^2 \left( \chi C_m + \tau K_M \right)}{\sigma_m^{\star (k)}} \left( 1 + \log \left( \frac{H}{h} \right) \right)^3,
		\end{eqnarray}
		where	$\sigma^{i,e}_M = \max_{\bullet = \left\{ l, t, n \right\}} \sigma^{i,e}_\bullet$, 	$ \sigma^{i,e}_m = \min_{\bullet = \left\{ l, t, n \right\}} \sigma^{i,e}_\bullet$, $P^{-1}$ denotes the FETI-DP or BDDC preconditioner, $Q$ is the system matrix, and $K_M$ is a constant independent of the subdomain diameter $H$ and mesh size $h$.
	\end{theorem}	
	
	The core of the proof relies on the following Lemma. 
	
	\begin{lemma}\label{edgelemma}
		Let $\chi C_m + \tau \pIon \geq c > 0$ with $c \in \mathbb{R}^+$. Let the primal set be spanned by the vertex nodal finite element functions and the subdomain edges averages. If the projection operator is scaled by the deluxe scaling, then
		$\forall u \in \widetilde{W}_\Gamma$
		\begin{equation*}
			\seminormS{\Ed u}^2 \lesssim  \max_{\substack{ k = 1,\dots,N \\ \star = i,e}} \frac{\tau \sigma_M^{\star (k)} + H^2 \left( \chi C_m + \tau K_M \right)}{\tau \sigma_m^{ \star (k)}} \left( 1 + \log \frac{H}{h} \right)^3 \seminormS{u}^2,
		\end{equation*} where
			$		\sigma^{i,e}_M = \max_{\bullet = \left\{ l, t, n \right\}} \sigma^{i,e}_\bullet, 	\ \ \ \ \sigma^{i,e}_m = \min_{\bullet = \left\{ l, t, n \right\}} \sigma^{i,e}_\bullet,
			$
		and $K_M$ is a constant independent of the subdomain diameter $H$ and mesh size $h$.
	\end{lemma}
	
	\begin{proof}
		Instead of proving the bound for the projection operator $\Ed$, we prove it for the complementary projection $\Pd$ (\ref{PdDeluxe}). Moreover, it is sufficient to compute only the local bounds, as it holds 	$\seminormS{\Pd u}^2 = \sumJN{ \seminormSj{R_{\pO_j} \Pd u }^2 } $.
		Thus, for all $u \in \widetilde{W}_\Gamma$ 
		\begin{equation*}
			\seminormSj{R_{\pO_j} \Pd u}^2 \leq | \varXi_j^\ast | \sum_{\substack{\ast = \{ \mathcal{F}, \mathcal{E} \}\\\ast \in \varXi_j^\ast}} \seminormSj{ R_\ast^T \left( u^{i,e}_{j, \ast} - \bar{u}^{i,e}_{\ast}  \right)}^2   , 
		\end{equation*}
		where $\varXi_j^\ast$ is the index set containing the indices of the subdomains that share the face $\mathcal{F}$ or the edge $\mathcal{E}$. Let us distinguish between face and edge contributions.
		
		{\bf Face contributions.} Suppose that the face $\mathcal{F}$ is shared by subdomains $\Omega_j$ and $\Omega_k$. Then, by simple algebra, it follows
		\begin{eqnarray} \nonumber
			\uiejF - \umeanFie 
			= \left( \SjF + \SkF \right)^{-1} \SkF \left[ \left( \uiejF - \umeanieFj \right) - \left( \uiekF - \umeanieFk \right)  + \left( \umeanieFj - \umeanieFk \right)\right],
		\end{eqnarray}
		where we add and subtract the mean value $\bar{u}^{i,e}_{\cdot, \mathcal{F}}$ of $u$ over $\mathcal{F}$ on the subdomain $\cdot$. Therefore, by noticing that $\RF \Sj_\Gamma \RF^T = \SjF$, it follows
		\begin{equation*}
			\seminormSj{\RF^T \left(\uiejF - \umeanFie\right)}^2 \leq  2 \seminormSjF{\uiejF - \umeanieFj}^2 + 2 \seminormSkF{\uiekF - \umeanieFk}^2 + \seminormSjF{\left( \SjF + \SkF \right)^{-1} \SkF \left( \umeanieFj - \umeanieFk \right)}^2,
		\end{equation*}
		$\forall \uiejF \in \widetilde{W}_\Gamma$, where we take advantage of two inequalities arising from the generalized eigenvalue problem $\SjF \phi = \lambda \SkF \phi$ and by observing that all eigenvalues are strictly positive.   \\
		
		It is sufficient now to estimate $\seminormSjF{\uiejF - \umeanieFj}^2$ and $\seminormSjF{\left( \SjF + \SkF \right)^{-1} \SkF \left( \umeanieFj - \umeanieFk \right)}^2$; we highlight that, in case also the subdomain faces averages are included in the primal space, the latter is zero. Starting from the first term, we make use of the bilinear form associated to the Schur complement 
		\begin{align*}
			\seminormSjF{\uiejF - \umeanieFj}^2 &\leq \left( \chi C_m + \tau K_M \right) \norm{\harmext{j}{\left( u_{j} - \umeanieFj \right)}}{\Omega_j}^2 + \tau \sum_{\star = i,e}\sigma^{\star}_M \seminormHone{\harmext{j}{\left( u_{j, \mathcal{F}} - \bar{u}_{j, \mathcal{F}}^\star \right)}}{\Omega_j}^2 \\
			&\lesssim \left[ \tau \sigma_M^{i,e} + H^2 \left( \chi C_m + \tau K_M \right) \right] \seminormHhalf{u_{j, \mathcal{F}} - \umeanieFj}{\Gammaj}^2,
		\end{align*}
		applying the ellipticity Lemma \ref{ellipbound} and the Poincar\'e-Friedrichs inequality \ref{frieddisug} combined with the Trace theorem \ref{tracethm}. 
		As we are already taking in consideration the discrete restriction of a function $u_j \in \widetilde{W}_\Gamma$ on the face $\mathcal{F}$, the notations $u_j - \umeanieFj$ and $\interpol{\Theta_{\mathcal{F}} (u_j - \umeanieFj ) }$ are essentially the same.
		Therefore, it is possible to apply \cite[Lemma $4.26$]{toselli2006domain} and the Trace theorem to get
		\begin{equation*}
			\seminormSj{\uiejF - \umeanieFj}^2 \lesssim \left[ \tau \sigma_M^{i,e} + H^2 \left( \chi C_m + \tau K_M \right) \right] \left( 1 + \log \dfrac{H}{h} \right)^2 \seminormHone{ \harmext{j}{u_j} }{\Omega_j}^2.
		\end{equation*}
		Regarding the second term, let $\mathcal{E} \subset \partial \mathcal{F}$ be a primal edge, such that we can add and subtract $\uiemeanEj = \uiemeanEk$. Then, using the same inequalities from the generalized eigenvalue problem,
		\begin{equation*}
			\seminormSjF{\left( \SjF + \SkF \right)^{-1} \SkF \left( \umeanieFj - \umeanieFk \right)}^2 \leq 2 \ \seminormSjF{  \uiemeanEj - \umeanieFj  }^2 + 2 \ \seminormSkF{ \uiemeanEk - \umeanieFk },
		\end{equation*}
		It is sufficient now to esteem the first term on the right-hand side, as we can deal with the other in the same fashion. Combining the result of ellipticity (Lemma \ref{ellipbound}), the Poincar\'e-Friedrichs inequality and the Trace theorem, we get
		\begin{align*}
			\seminormSjF{  \uiemeanEj - \umeanieFj  }^2 &= \bilformj{\harmextie{j}{\left( \uiemeanEj - \umeanieFj \right) }}{\harmextie{j}{\left( \uiemeanEj - \umeanieFj \right)}} \\ 
			&= \bilformj{\harmextie{j}{ \overline{\left( u_j - \umeanieFj \right) }_{j, \mathcal{E}} }}{\harmextie{j}{ \overline{\left( u_j - \umeanieFj \right) }_{j, \mathcal{E}} }} \\
			&\overset{\substack{Lemma \\\ref{ellipbound}}}{\leq} \left( \chi C_m + \tau K_M \right) \norm{\harmext{j}{ \overline{\left( u_j - \umeanieFj \right) }_{j, \mathcal{E}} }}{\Omega_j}^2 \ + 
			\tau \sum_{\star = i,e}\sigma^{\star}_M \seminormHone{\harmext{j}{ \overline{\left( u_j - \umeanieFj \right) }_{j, \mathcal{E}} }}{\Omega_j}^2 \\
			&\lesssim  \sum_{\star = i,e} \left[ \tau \sigma^{\star}_M + H^2 \left( \chi C_m + \tau K_M \right) \right] \seminormHone{\overline{\left( u_j - \umeanieFj \right) }_{j, \mathcal{E}} }{\Omega_j}.
		\end{align*}
		Using \cite[Lemmas $4.16$, $4.17$ and $4.19$]{toselli2006domain} and \cite[Lemma $4.30$]{toselli2006domain}, it follows 
		\begin{equation*}
			\parallel \overline{\left( u_j - \umeanieFj \right) }_{j, \mathcal{E}} \parallel^2 \ \leq C \left( 1 + \log \dfrac{H}{h} \right)^{3} \seminormHhalf{u_j}{\pO_j}^2 .
		\end{equation*}
		This means that
		\begin{equation*}
			\seminormSjF{  \uiemeanEj - \umeanieFj  }^2 \lesssim C \left( 1 + \log \dfrac{H}{h} \right)^3 \sum_{\star = i,e} \left[ \tau \sigma^{\star}_M + H^2 \left( \chi C_m + \tau K_M \right) \right]  \seminormHone{\harmext{j}{u_j}}{\Omega_j}^2 .
		\end{equation*}
		To conclude, the face contribution gives the bound 
		\begin{align*}
			\seminormSj{P_D u}^2 &\lesssim \sum_{\star = i,e}  \sum_{\mathcal{F} \in \varXi_j^{\mathcal{F}}}  \left[ \tau \sigma_M^{\star} + H^2 \left( \chi C_m + \tau K_M \right) \right] \left( 1 + \log \dfrac{H}{h} \right)^3\seminormHone{ \harmext{j}{u_j} }{\Omega_j}^2.  \\
			&\leq \max_{\star = i,e}  \sum_{\mathcal{F} \in \varXi_j^{\mathcal{F}}}  \dfrac{ \tau \sigma_M^{\star} + H^2 \left( \chi C_m + \tau K_M \right) }{\tau \sigma_m^{\star}} \left( 1 + \log \dfrac{H}{h} \right)^3 \seminormSj{u_j}^2 \\
		\end{align*}	
		
		{\bf Edge contributions.} For simplicity, suppose that an edge $\mathcal{E}$ is shared only by three substructures, each with indexes $j_1$, $j_2$ and $j_3$. The extension to the case of more subdomains is then similar. Define $\SjallE := \SjunoE + \SjdueE + \SjtreE$.
		Then, the average operator is given by
		\begin{equation*}
			\uiemeanE := \left( \SjallE \right)^{-1} \left( \SjunoE \uEjuno + \SjdueE \uEjdue + \SjtreE \uEjtre \right).
		\end{equation*}
		Proceeding in the same fashion as for the face contribution, it follows
		\begin{equation*}
			\uEjuno - \uiemeanE = \left( \SjallE \right)^{-1} \left[ \left( \SjdueE + \SjtreE \right) \uEjuno - \SjdueE \uEjdue - \SjtreE \uEjtre \right],
		\end{equation*}
		which leads to
		\begin{equation*}
			\seminormSjuno{\RE^T \left(\uEjuno - \uiemeanE \right)}^2 \leq 3 \uEjuno^T \SjunoE \uEjuno \ + 3 \uEjdue^T \SjdueE \uEjdue \ + \uEjtre^T  \SjtreE \uEjtre ,
		\end{equation*} 
		where we use analogous inequalities as in the face case. 
		
		Since we have included the edge averages into the primal space, we have the same average value for the three subdomains.
		By adding and subtracting the mean value over the edges $\uiemeanEjuno$, we can get the 
		estimate for the edges by using \cite[Lemmas $4.16$, $4.17$ and $4.19$]{toselli2006domain}:
		\begin{equation*}
			\uEjuno^T \SjunoE \uEjuno \leq \left[ \tau \sigma_M^{i,e} + H^2 \left( \chi C_m + \tau K_M \right) \right] \left( 1 + \log \dfrac{H}{h} \right) \seminormHone{\harmext{j_1}{\uEjuno}}{\Omega_j}^2.
		\end{equation*}
		In conclusion, the edge estimate gives
		\begin{eqnarray}\nonumber
			\seminormSj{P_D u}^2 \leq \max_{\star = i,e}  \sum_{\mathcal{E} \in \varXi_j^\mathcal{E}} \dfrac{ \tau \sigma_M^{\star} + H^2 \left( \chi C_m + \tau K_M \right) }{\tau \sigma_m^\star} \left( 1 + \log \dfrac{H}{h} \right)  \seminormSjuno{u_{j}}^2,
		\end{eqnarray}
		where the index $j$ collects all contributions from the subdomains that share the edge $\mathcal{E}$.
	\end{proof}

	\section{Parallel numerical results}	\label{results}
	We report here the results of several parallel numerical tests which confirm our theoretical estimates and study the performance of the proposed preconditioners with respect to the discretization parameters. 
	
	The weak scaling tests (with fixed local problem size per processor while the total problem size increases with the processor count) are performed on the supercomputer Galileo from Cineca centre, a Linux Infiniband cluster equipped with 1084 nodes, each with 36 2.30 GHz Intel Xeon E5-2697 v4 cores and 128 GB/node, for a total of 39024 cores. Instead, the strong scaling tests (with fixed total problem size while the local problem size per processor decreases with the inverse of the processor count) are computed on cluster Indaco at the University of Milan, a Linux Infiniband cluster with 16 nodes, each carrying 2 processors Intel Xeon E5-2683 V4 2.1 GHz with 16 cores each, for a total amount of 512 cores. The optimality tests are carried out on the cluster Eos at University of Pavia, a Linux Infiniband cluster with 21 nodes, each carrying 2 processors Intel Xeon Gold 6130 2.1 GHz with 16 cores each, for a total of 672 cores. Our C code is based on PETSc library (\cite{balay2019petsc}) from Argonne National Laboratory. 
	
	In tests 1, 2 and 4 dual-primal preconditioners are applied with the standard rho-scaling, as test 3 shows an almost computational equivalence between rho- and deluxe scaling for this application model, as concerns for non-linear and linear iterations numbers per time step. 
	
	\begin{figure}[!h]
		\begin{center}
			\begin{multicols}{3}
				t = 10 ms 	\vspace*{1mm} \\	\includegraphics[scale=.18]{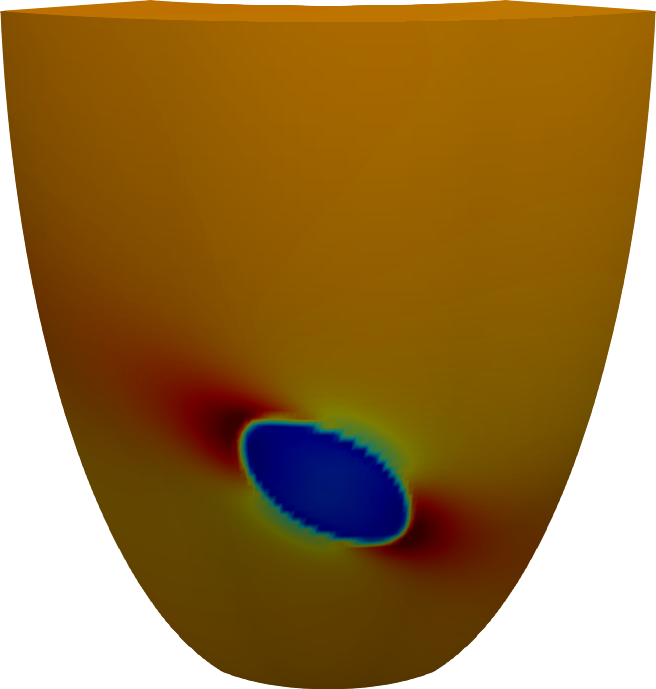} \\	\vspace*{1mm}
				t = 25 ms 	\vspace*{1mm} \\	\includegraphics[scale=.18]{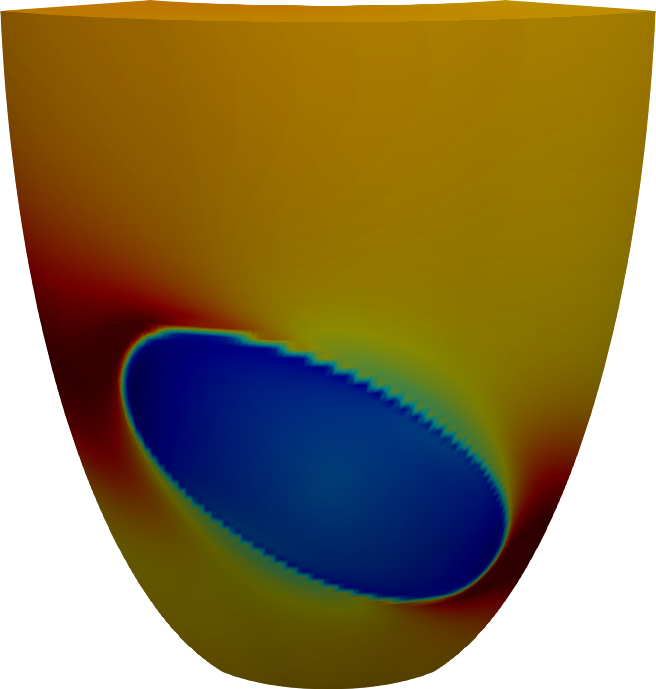} \\	\vspace*{1mm}
				t = 40 ms 	\vspace*{1mm} \\	\includegraphics[scale=.18]{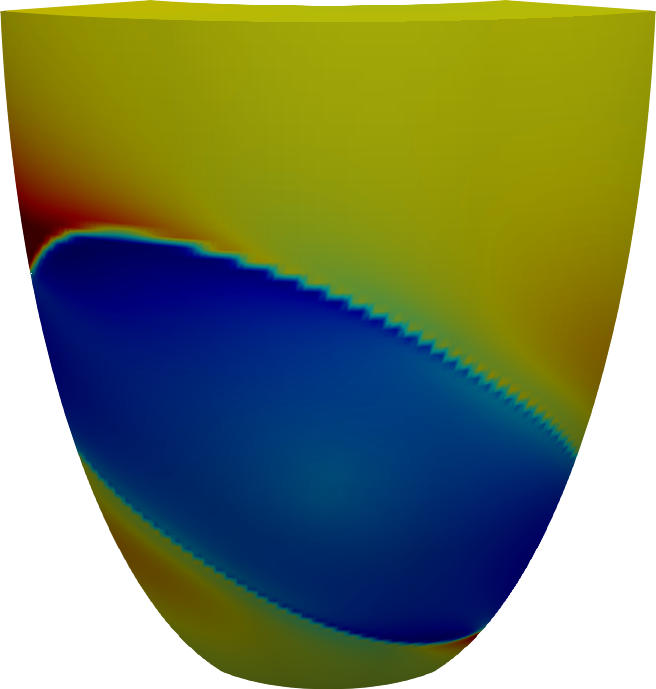} \\
				\columnbreak
				t = 15 ms 	\vspace*{1mm} \\	\includegraphics[scale=.18]{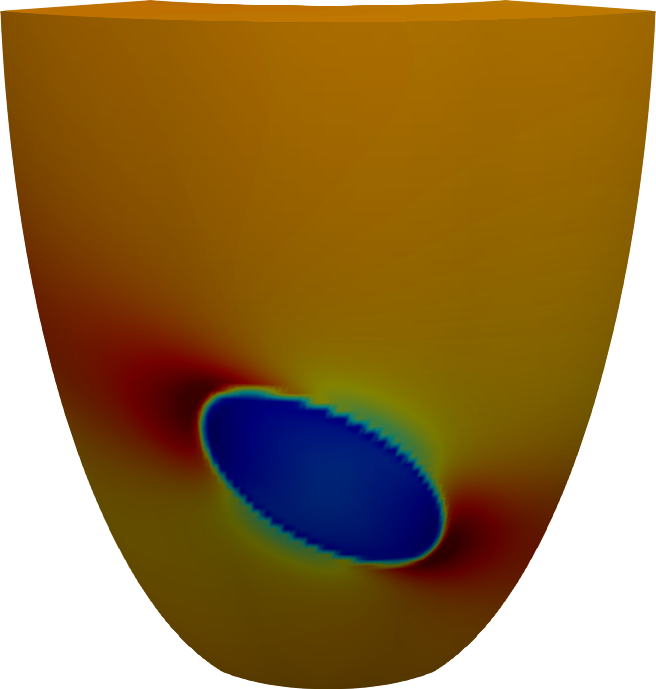} \\	\vspace*{1mm}
				t = 30 ms 	\vspace*{1mm} \\	\includegraphics[scale=.18]{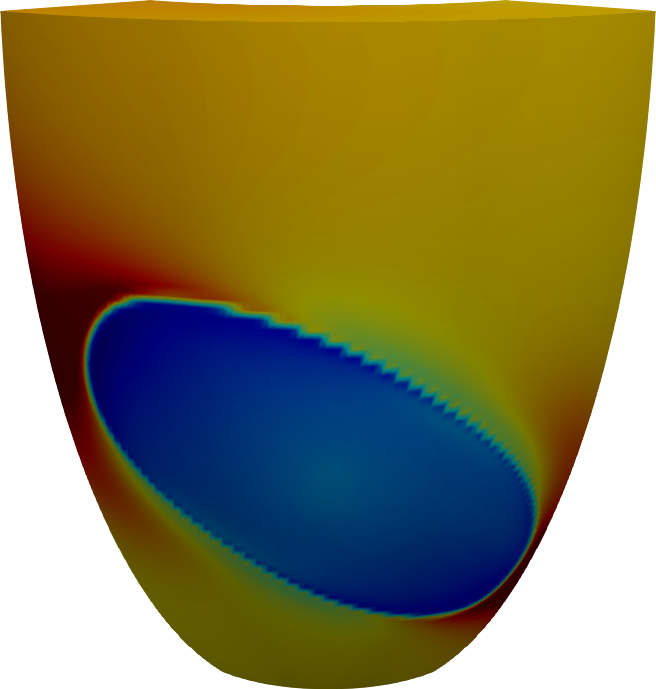} \\	\vspace*{1mm}
				t = 45 ms 	\vspace*{1mm} \\	\includegraphics[scale=.18]{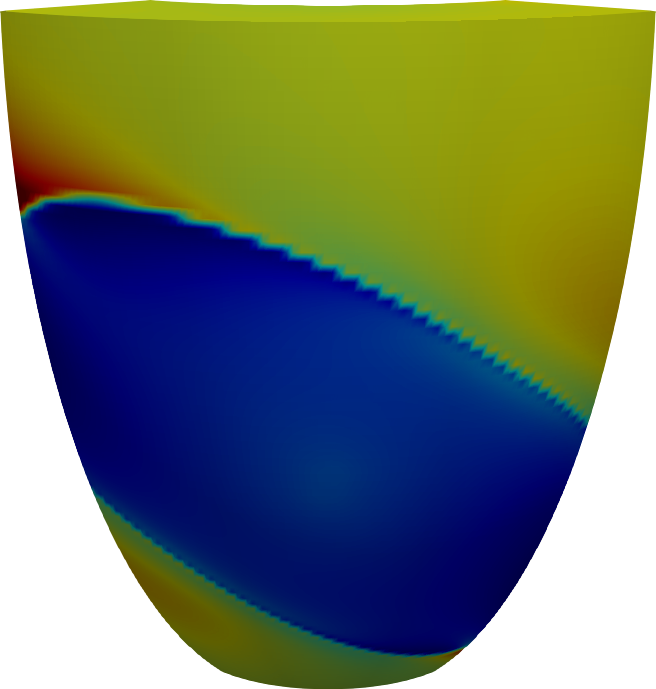} \\
				\columnbreak
				t = 20 ms 	\vspace*{1mm} \\	\includegraphics[scale=.18]{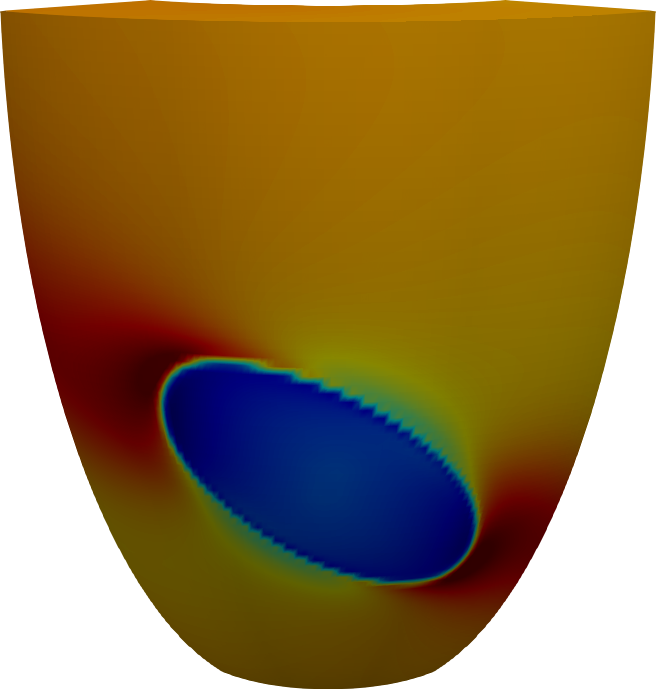} \\	\vspace*{1mm}
				t = 35 ms 	\vspace*{1mm} \\	\includegraphics[scale=.18]{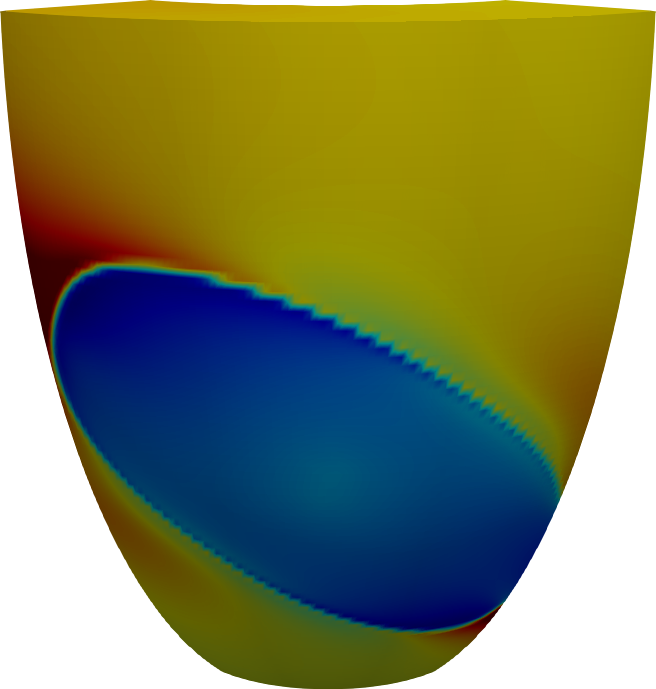} \\	\vspace*{1mm}
				t = 50 ms 	\vspace*{1mm} \\	\includegraphics[scale=.18]{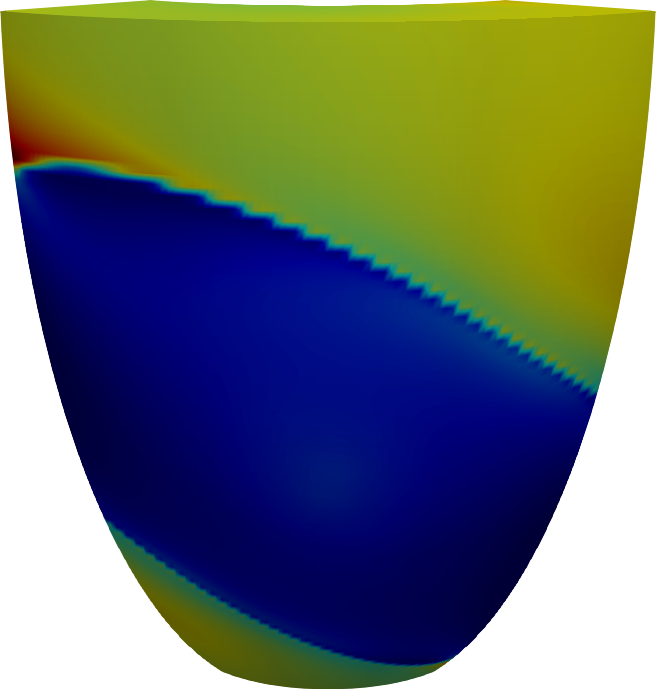}
			\end{multicols}
			\includegraphics[scale=.3]{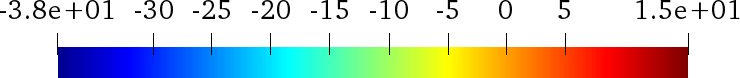}
		\end{center}
		\caption{Snapshots (every 5 ms) of extra-cellular potential $u_e$ time evolution. For each time frame, we report the epicardial view of a portion of the left ventricle, modeled as a truncated ellipsoid. }
		\label{fig_bido_snapshots_ue}
	\end{figure}
	\begin{figure}[!h]
		\begin{center}
			\begin{multicols}{3}
				t = 10 ms 	\vspace*{1mm} \\	\includegraphics[scale=.18]{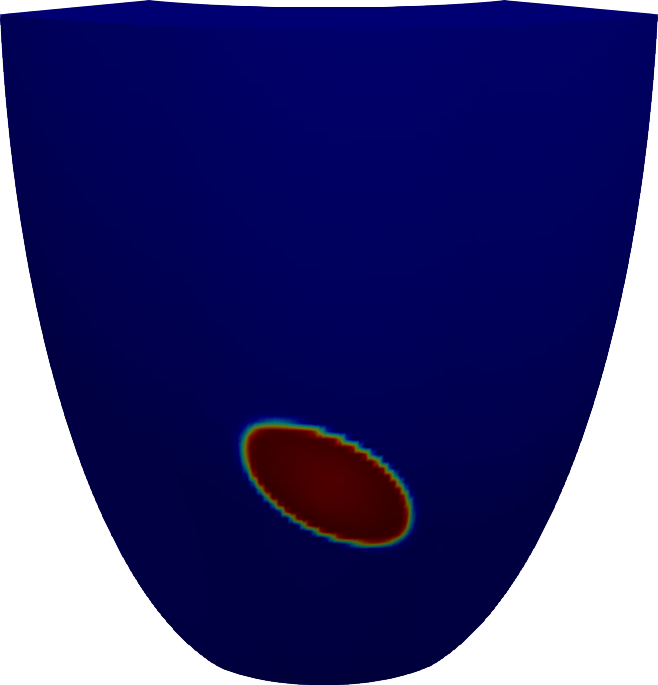} \\	\vspace*{1mm}
				t = 25 ms 	\vspace*{1mm} \\	\includegraphics[scale=.18]{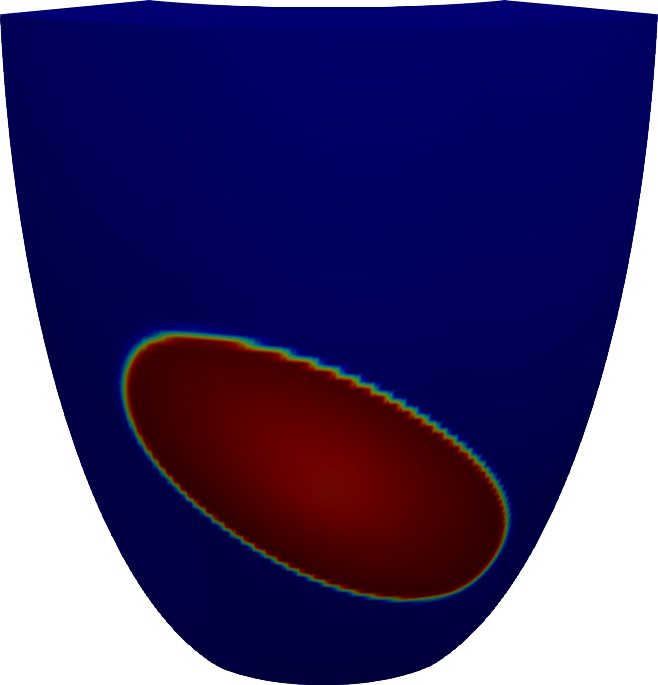} \\	\vspace*{1mm}
				t = 40 ms 	\vspace*{1mm} \\	\includegraphics[scale=.18]{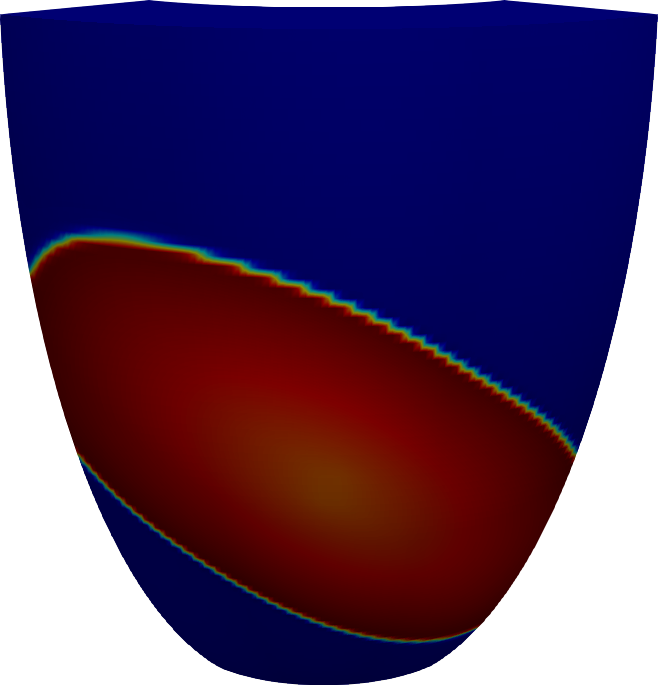} \\
				\columnbreak
				t = 15 ms 	\vspace*{1mm} \\	\includegraphics[scale=.18]{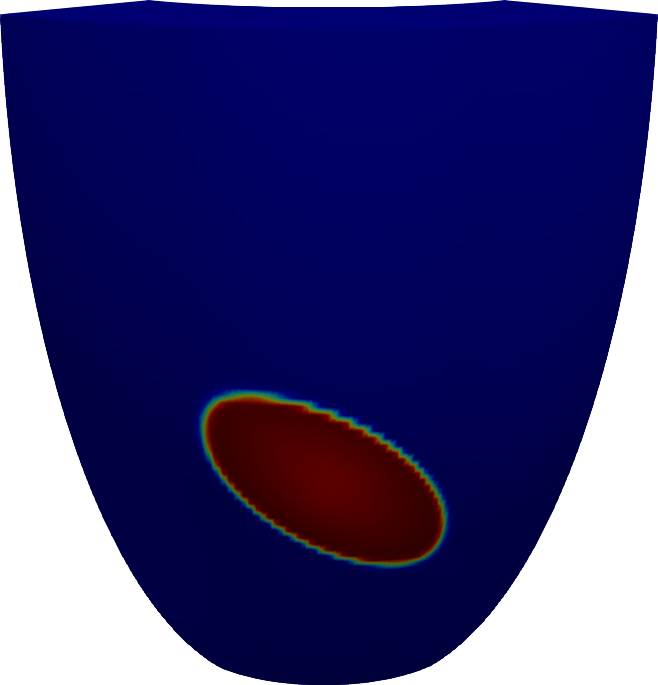} \\	\vspace*{1mm}
				t = 30 ms 	\vspace*{1mm} \\	\includegraphics[scale=.18]{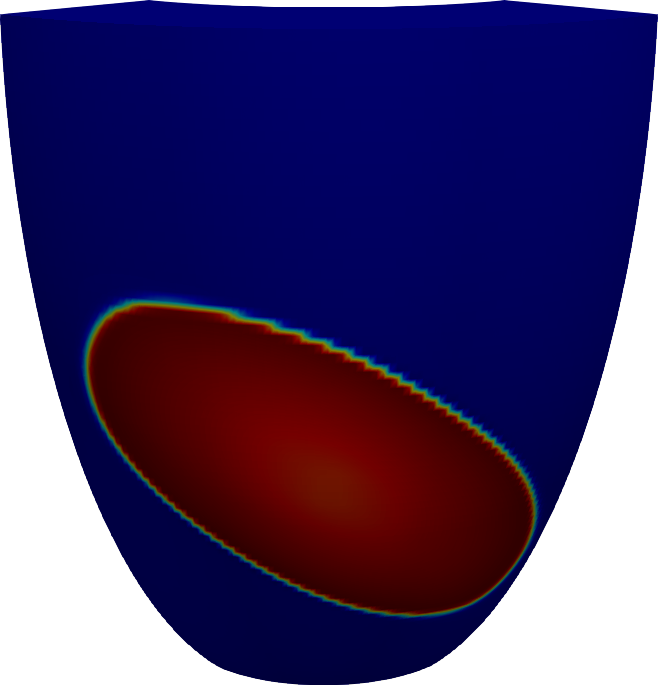} \\	\vspace*{1mm}
				t = 45 ms 	\vspace*{1mm} \\	\includegraphics[scale=.18]{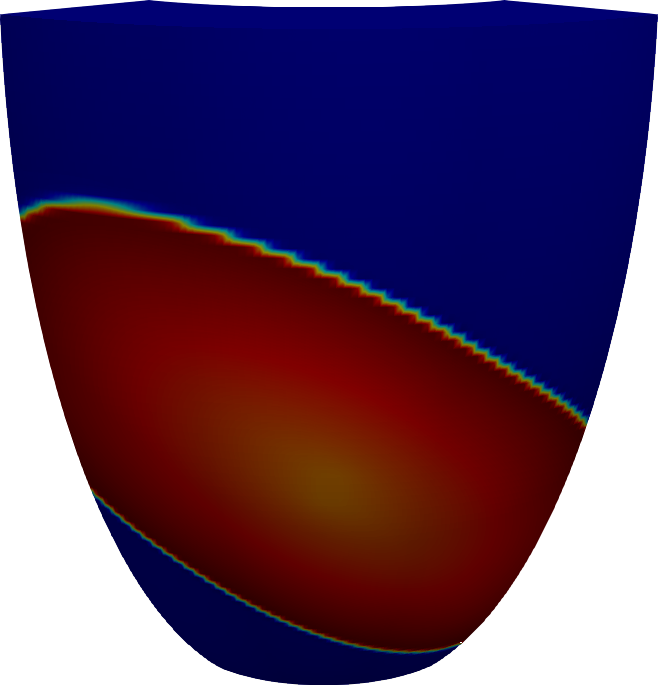} \\
				\columnbreak
				t = 20 ms 	\vspace*{1mm} \\	\includegraphics[scale=.18]{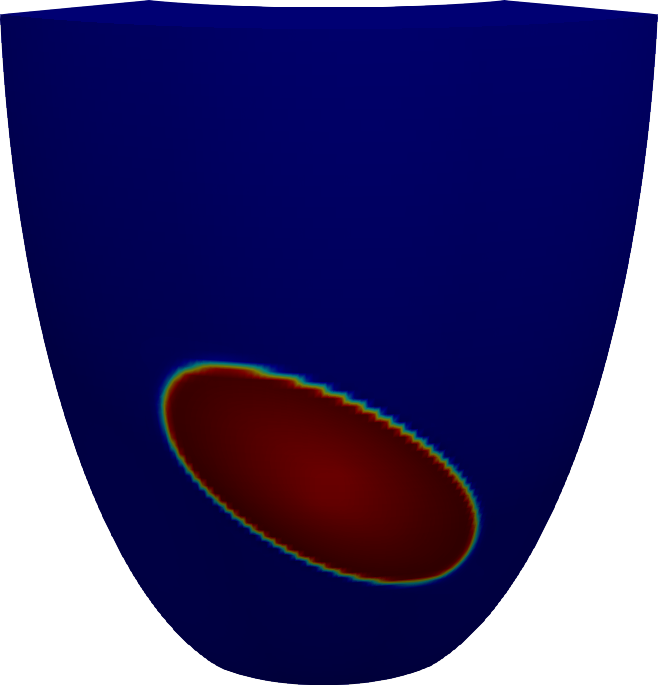} \\	\vspace*{1mm}
				t = 35 ms 	\vspace*{1mm} \\	\includegraphics[scale=.18]{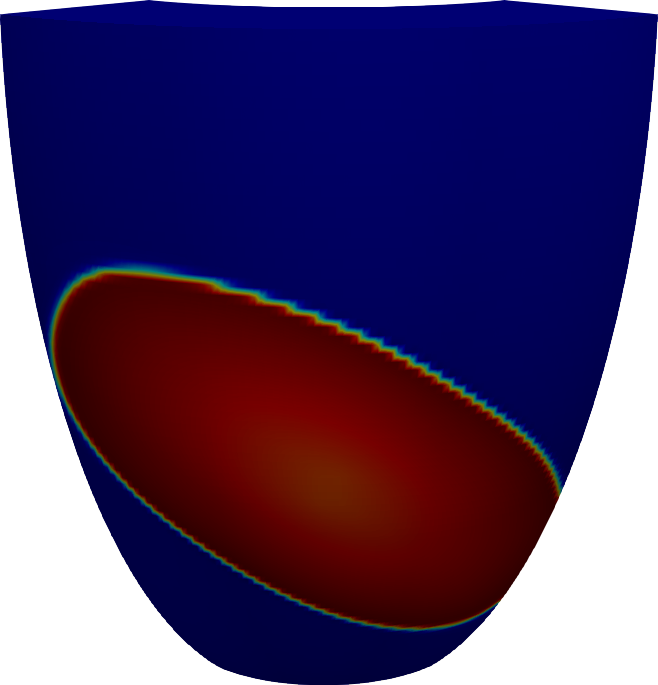} \\	\vspace*{1mm}
				t = 50 ms 	\vspace*{1mm} \\	\includegraphics[scale=.18]{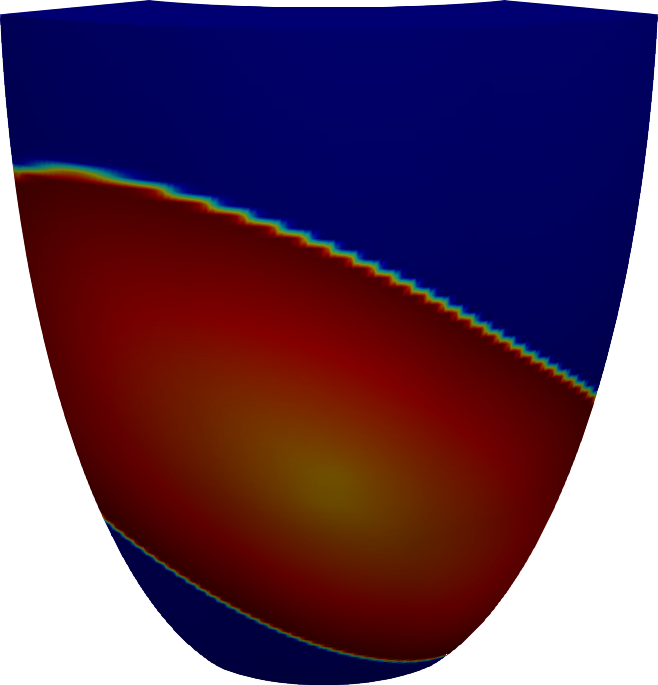}
			\end{multicols}
			\includegraphics[scale=.3]{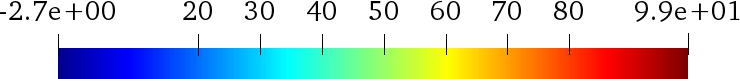}
		\end{center}
		\caption{Snapshots (every 5 ms) of transmembrane potential $v$ time evolution. For each time frame, we report the epicardial view of a portion of the left ventricle, modeled as a truncated ellipsoid. }
		\label{fig_bido_snapshots_v}
	\end{figure}
	
	All our numerical experiments are carried both on a thin slab and on an idealized left ventricular geometry, modeled as a truncated ellipsoid. 
	The latter is described in ellipsoidal coordinates by the parametric equations
	\begin{eqnarray}
		\begin{dcases}
			\mathbf{x} = a (r) \cos \theta \cos \varphi,		&\theta_\text{min} \leq \theta \leq \theta_\text{max},		\nonumber \\
			\mathbf{y} = b (r) \cos \theta \sin \varphi, 		&\varphi_\text{min} \leq \varphi \leq \varphi_\text{max}, 	\nonumber  \\
			\mathbf{z} = c (r) \sin \varphi ,	\qquad			 &0 \leq r \leq 1,		\nonumber 
		\end{dcases}
	\end{eqnarray}
	where $a (r)  = a_1 + r (a_2 - a_1)$, $b (r)  = b_1 + r (b_2 - b_1)$ and $c (r)  = c_1 +r (c_2 - c_1)$ with $a_{1,2}$, $b_{1,2}$ and $c_{1,2}$ given coefficients defining the main axes of the ellipsoid.  
	
	The fibers rotate intramurally linearly with the depth for a total amount of $120^o$ proceeding counterclockwise from epicardium ($r = 1$, outer surface of the truncated ellipsoid) to endocardium ($r=0$, inner surface). 
	Regarding the physiological coefficients in Table \ref{physiolcoeff}, we refer to the original paper \cite{rogers1994collocation} for the parameters of the ionic membrane model, while we refer to \cite{franzone2014mathematical} for the Bidomain and Monodomain parameters.
	
	\begin{table}[H]
		\centering
		\begin{tabular}{*{5}{r}}
			\hline
			\multicolumn{2}{r}{Bidomain conductivity coeff.}	& &\multicolumn{2}{c}{Ionic parameters}	\\
			\cline{1-2} \cline{4-5} 
			$\sigma_l^i$		  & $3 \times 10^{-3} \Omega^{-1} \text{ cm}^{-1}$ 				&  &G 				 &$1.2$ $\Omega^{-1}$ cm$^{-2}$	\\
			$\sigma_t^i$		 & $3.1525 \times 10^{-4} \Omega^{-1} \text{ cm}^{-1}$		& &$\eta_1$	  	&$4.4$ $\Omega^{-1}$ cm$^{-1}$	\\
			$\sigma_n^i$   		& $3.1525 \times 10^{-5} \Omega^{-1} \text{ cm}^{-1}$	    &  &$\eta_2$  	 &$0.012$	 								\\
			$\sigma_l^e$		& $2 \times 10^{-3} \Omega^{-1} \text{ cm}^{-1}$ 				&  &$v_{th}$	 &13 mV  \\
			$\sigma_t^e$	    & $1.3514 \times 10^{-3} \Omega^{-1} \text{ cm}^{-1}$		&  &$v_p$		  &100 mV \\
			$\sigma_n^e$	   & $6.757 \times 10^{-4} \Omega^{-1} \text{ cm}^{-1}$			&	&$C_m$ 		&$1$ mF/cm$^2$ 		\\
			\hline
		\end{tabular}
		\caption{Conductivity coefficients for the Bidomain model and physiological parameters for the Rogers-McCulloch ionic model.}
	\label{physiolcoeff}
	\end{table}
	
	The external stimulus of $\Iapp = 100$ mA/cm$^3$, needed for the potential to start propagating, is applied for $1$ ms to the surface of the domain representing the endocardium. Instead, if a slab geometry is considered, the stimulus is applied in one corner of the domain, over a spheric volume of radius $0.1$ cm. \\ Figures \ref{fig_bido_snapshots_ue} and \ref{fig_bido_snapshots_v} show the time evolution of the extra-cellular and transmembrane potentials respectively, from the epicardial view of a portion of the idealized left ventricle when the external stimulus $\Iapp$ is applied at an epicardial location.
	We consider insulating boundary conditions, resting initial conditions, and a fixed time step size $\tau = 0.05$ ms.
	
	In order to test the efficiency of our solver on parallel architectures, we also compute the parallel speedup $S_N = \frac{T_1}{T_N}$, the ratio between the runtime $T_1$ on 1 processor and the average runtime $T_N$ on $N$ processors.

	We use the default non-linear solver (SNES) in the PETSc library \cite{balay2019petsc}, which consists in a Newton method with cubic backtracking linesearch. We adopt the default SNES convergence test as stopping criterion, based on the comparison of the $L^2$-norm of the non-linear function at the current iterate and at the current step with specified tolerances.
	Since the linear system arising from the discretization of the Jacobian problem at each Newton step is symmetric, we solve it with the Preconditioned Conjugate Gradient (PCG) method, with the BDDC or FETI-DP preconditioners from the PETSc library or with the {\it Boomer} Algebraic MultiGrid (bAMG) preconditioner from the Hypre library.  
	We use the default Hypre parameters strong threshold = 0.25, number of smoothing levels = 25 and  number of levels of aggressive coarsening = 0.
	In the strong scaling tests, when testing the performance of the proposed solver against two different ionic models (Rogers-McCulloch and Luo-Rudy phase 1), the Generalized Minimal Residual (GMRES) method is applied. The convergence criteria of the linear solver is based on the decreasing of the residual norm (default from PETSc). 
	
	All parameters can be found in Table \ref{tolerances}.\\
	
	\begin{table}[!h]
		\label{tolerances}
		\centering
		\begin{tabular}{*{4}{r}}
			\hline
			KSP     & $r_{tol} = 1e^{-08}$     & $a_{tol} = 1e^{-10}$     & $d_{tol}= 1e^{+04}$        \\
			SNES    & $r_{tol} = 1e^{-04}$     & $a_{tol} = 1e^{-08}$     & $s_{tol}= 1e^{-08}$        \\
			\hline
		\end{tabular}
	\caption{PETSc SNES and KSP tolerances. $r_{tol}$ is the relative convergence tolerance, $a_{tol}$ is the absolute convergence tolerance, $d_{tol}$ is the KSP divergence convergence tolerance and $s_{tol}$ is the convergence tolerance related to the solution change between Newton steps. }
	\end{table}

	\paragraph{\bf Test 1: weak scaling.}
	The first set of tests we report here is a weak scaling test on both slab and ellipsoidal domain, performed on Galileo cluster. For both cases, we fix the local mesh size to $16\cdot 16\cdot 16$ and we increase the number of subdomains from $32$ to $2048$, thus resulting in an increasing slab geometry and in an increasing portion of ellipsoid. 
	From Tables \ref{table_bido_weak_slab} and \ref{table_bido_weak_ell}, it is evident how the dual-primal algorithms have a better performance than bAMG: the average number of linear iteration per Newton iteration (lit) is clearly lower and does not increase with the number of subdomains, except for BDDC on the slab domain, where the linear iterations increase unexpectedly. Moreover, the reported average CPU times (in seconds) per Newton step (see also Fig. \ref{fig_bido_weak_time}) are slightly better for the BDDC and FETI-DP preconditioners, except for FETI-DP on the slab domain and 2048 processors. 
	In the harder ellipsoidal tests, both BDDC and FET-DP are scalable and outperform bAMG when the number of processors increases past 128, indicating lower computational complexity and interprocessor communications. 
	
	\begin{table}[h]
		\centering
		\begin{tabular}{*{16}{r}}
			\hline
			\multirow{2}{*}{procs}	&& \multirow{2}{*}{mesh}				& \multirow{2}{*}{dofs}			&& \multicolumn{3}{c}{bAMG}	&& \multicolumn{3}{c}{BDDC}	&& \multicolumn{3}{c}{FETI-DP}\\
			&& 				& 		&& nit 	& lit	& time		&& nit	& lit	& time		&& nit	& lit	& time\\
			\cline{1-1} \cline{3-4} \cline{6-8} \cline{10-12} \cline{14-16}
			32		&& $64\cdot 64\cdot 32$  	 & 278,850		&& 1.25	& 106	& 4.9		&& 1.0	& 22	& 6.1		&& 1.25	& 10	& 6.0\\
			64		&& $128\cdot 64\cdot 32$	 & 553,410		&& 1.25	& 132	& 6.8		&& 1.0	& 27	& 6.2		&& 1.25	& 11	& 6.0\\
			128		&& $128\cdot 128\cdot 32$	& 1,098,306	&& 1.25	& 180	& 9.4		&& 1.0	& 32	& 7.6		&& 1.25	& 10	& 7.4\\
			256		&& $256\cdot 128\cdot 32$	& 2,188,098	&& 1.25	& 237	& 15.2		&& 1.0	& 39	& 7.2		&& 1.25	& 10	& 7.9\\
			512		&& $256\cdot 256\cdot 32$	& 4,359,234	&& 1.25	&318 	& 20.0		&& 1.0	& 48	& 10.1		&& 1.25	& 10	& 11.1\\
			1024	&& $512\cdot 256\cdot 32$	& 8,701,506	&& 1.25	&405	& 29.6		&& 1.0	& 63	& 13.8		&& 1.25	& 10	& 18.7\\
			2048	&& $512\cdot 512\cdot 32$	& 17,369,154	&& 1.25	&536	& 40.3		&& 1.0	& 78	& 33.5		&& 1.25	& 10	& 63.2\\
			\hline
		\end{tabular}	
		\caption{Weak scaling test for the Bidomain decoupled solver on the cluster Galileo.  Slab domain, local mesh $16\cdot 16\cdot 16$ elements. Simulations of 2 ms of cardiac activation with $dt = 0.05$ ms (40 time steps). Comparison of Newton-Krylov solvers with boomerAMG (bAMG), BDDC and FETI-DP preconditioners. Average Newton iterations per time step (nit); average conjugate gradient iterations per Newton iteration (lit); average CPU solution time per time step (time) in seconds.}
	\label{table_bido_weak_slab}
		\vspace{5mm}
		\centering
		\begin{tabular}{*{16}{r}}
			\hline
			\multirow{2}{*}{procs}	&& \multirow{2}{*}{mesh}				& \multirow{2}{*}{dofs}			&& \multicolumn{3}{c}{bAMG}	&& \multicolumn{3}{c}{BDDC}	&& \multicolumn{3}{c}{FETI-DP}\\
			&& 				& 		&& nit 	& lit	& time		&& nit	& lit	& time		&& nit	& lit	& time\\
			\cline{1-1} \cline{3-4} \cline{6-8} \cline{10-12} \cline{14-16}
			32	&& $64\cdot 32\cdot 64$  	& 278,850	&& 1.0	    & 86	& 3.3		&& 1.0	& 30	& 5.4		&& 1.0	& 20	& 4.7\\
			64	&& $64\cdot 64\cdot 64$		& 549,250	&& 1.07	    & 124	& 6.0		&& 1.07	& 37	& 6.2		&& 1.07	& 20	& 6.5\\
			128	&& $64\cdot 128\cdot 64$	& 1,090,050 && 1.20	& 207	& 11.3		&& 1.20	& 26	& 7.5		&& 1.2	& 19	& 6.6\\
			256	&& $64\cdot 256\cdot 64$	& 2,171,650 && 1.42	& 348	& 22.2		&& 1.42	& 25	& 8.7		&& 1.42	& 17	& 10.7\\
			512	&& $128\cdot 256\cdot 64$	& 4,309,890 && 1.42	& 335	& 21.3		&& 1.42	& 27	& 10.5		&& 1.42	& 18	& 11.4\\
			1024 && $256\cdot 256\cdot 64$	& 8,586,370 && \multicolumn{3}{c}{\small out of memory}	&& 1.42	& 28	& 12.5		&& 1.42	& 19	& 11.0\\
			2048 && $512\cdot 256\cdot 64$	& 17,139,330 && \multicolumn{3}{c}{\small out of memory}	&& 1.42	& 28	& 26.6		&& 1.42	& 19	& 21.4\\
			\hline
		\end{tabular}
		\caption{Weak scaling test for the Bidomain decoupled solver on the cluster Galileo.  Ellipsoidal domain, local mesh $16\cdot 16\cdot 16$ elements. Simulations of 2 ms of cardiac activation with $dt = 0.05$ ms (40 time steps). Comparison of Newton-Krylov solvers with boomerAMG (bAMG), BDDC and FETI-DP preconditioners. Average Newton iterations per time step (nit); average conjugate gradient iterations per Newton iteration (lit); average CPU solution time per time step (time) in seconds.}
	\label{table_bido_weak_ell}
	\end{table}
	
	\begin{figure}[!h]
		\centering
		\includegraphics[scale=.45]{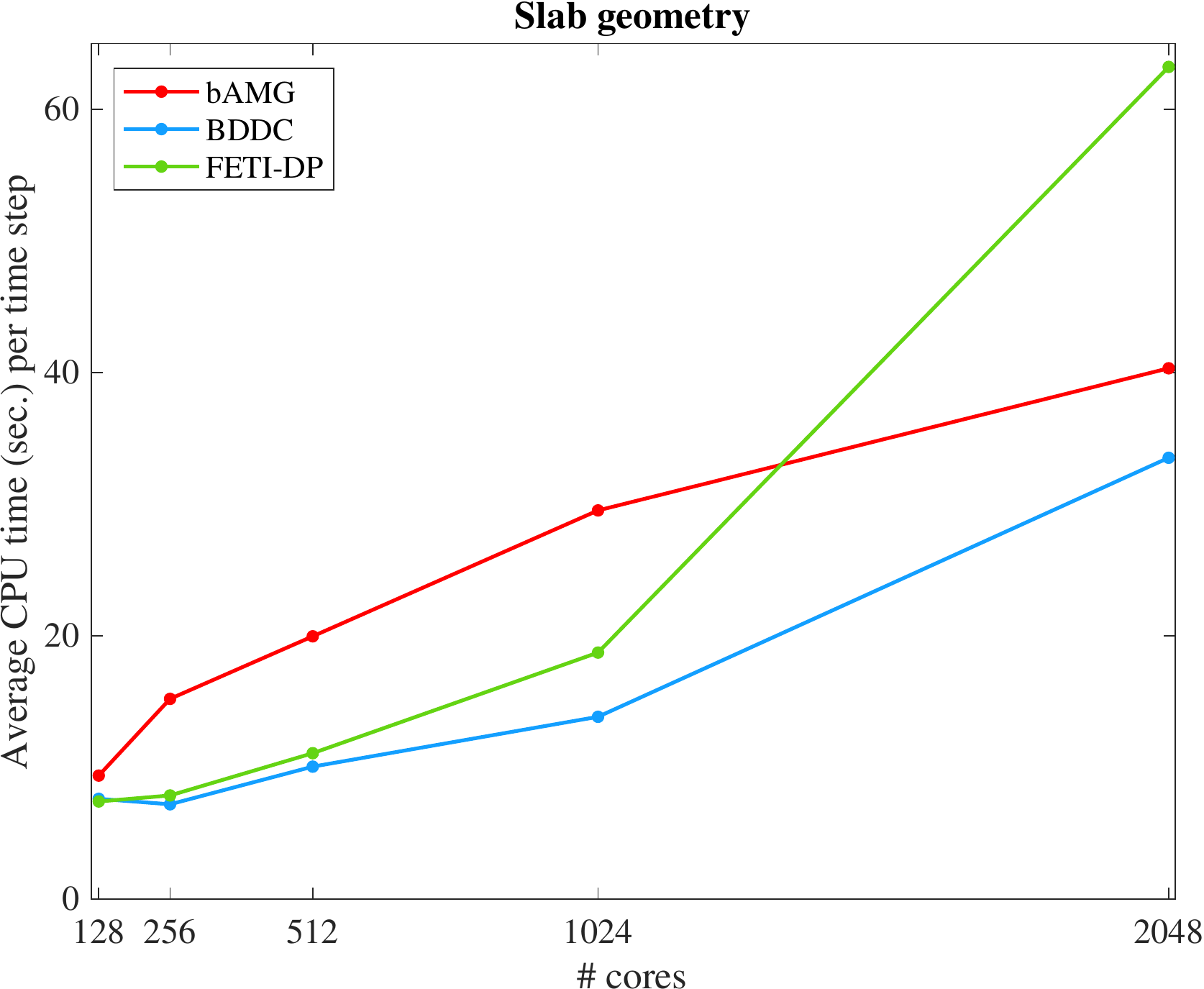}
		\ \ \ \ \  
		\includegraphics[scale=.45]{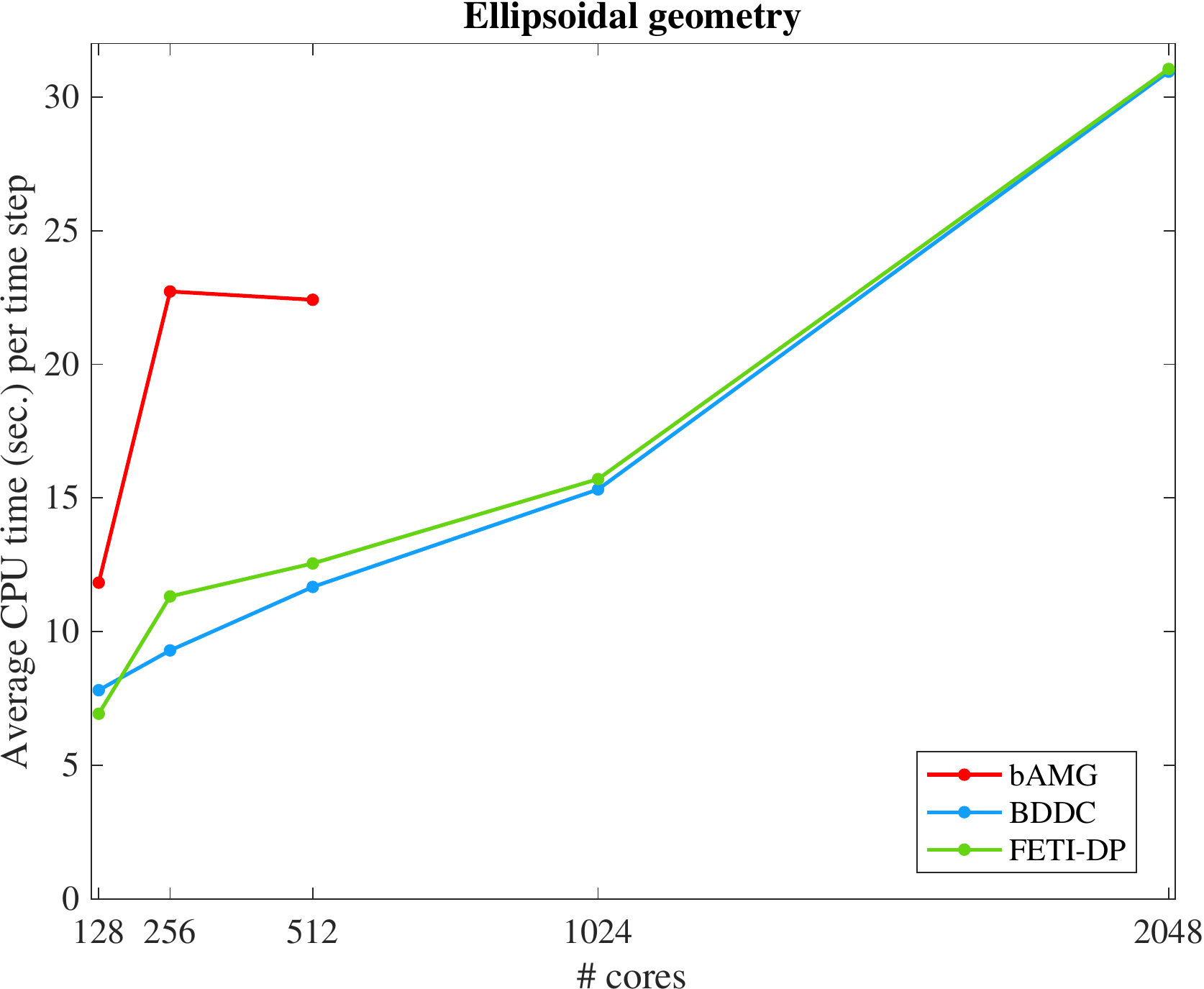}
		\caption{Weak scaling test on the cluster Galileo. Slab (left) and ellipsoidal (right) domains, local mesh $16\cdot 16\cdot 16$ elements. Simulations of 2 ms of cardiac activation with $dt = 0.05$ ms (40 time steps). Comparison of average CPU time per time step, in seconds.}
		\label{fig_bido_weak_time}
	\end{figure}

	\paragraph{\bf Test 2: strong scaling.} 
	We now perform a strong scaling test for the two geometries on Indaco cluster.
	For the thin slab geometry, we fix the global mesh to $192 \cdot 192 \cdot 32$ elements and we increase the number of subdomains. We fix the global mesh to $ 128 \cdot 128 \cdot 64$ elements for the portion of ellipsoid instead.
	We observe from Tables \ref{table_bido_strong_slab_indaco} and \ref{table_bido_strong_ell_indaco}  that, as the local number of dofs decrease, the preconditioner with the better balance in term of average linear iterations and CPU time per time step is FETI-DP.
	In both cases, BDDC and FETI-DP preconditioners outperforms the ideal speedup, while bAMG is sub-optimal (see Fig. \ref{fig_bido_speedup}).  
	Moreover we compare the performance of the Newton-Krylov solver with BDDC preconditioner using the Rogers-McCulloch (RMC) and Luo-Rudy phase 1 (LR1) ionic models in Tables \ref{table_bido_strong_slab_comparison_indaco} and \ref{table_bido_strong_ell_comparison_indaco}. In this case the Jacobian linear system is solved with the GMRES method.   \\
	By increasing the complexity of the ionic current, we observe an increasing in the average number of Newton iterations from 1-2 per time step using the RMC model to 2-3 per time step with the LR1 model. On the other hand, the average numbers of linear iterations per time step for the two ionic models are comparable, indicating that our dual-primal solver retains its good convergence properties even for more complex ionic models. As a consequence, the CPU times for the LR1 model increase due to the increase of nonlinear iterations, but the associated parallel speedups of the two models are comparable.  
	
	\begin{figure}[!h]
		\centering
		\includegraphics[scale=.45]{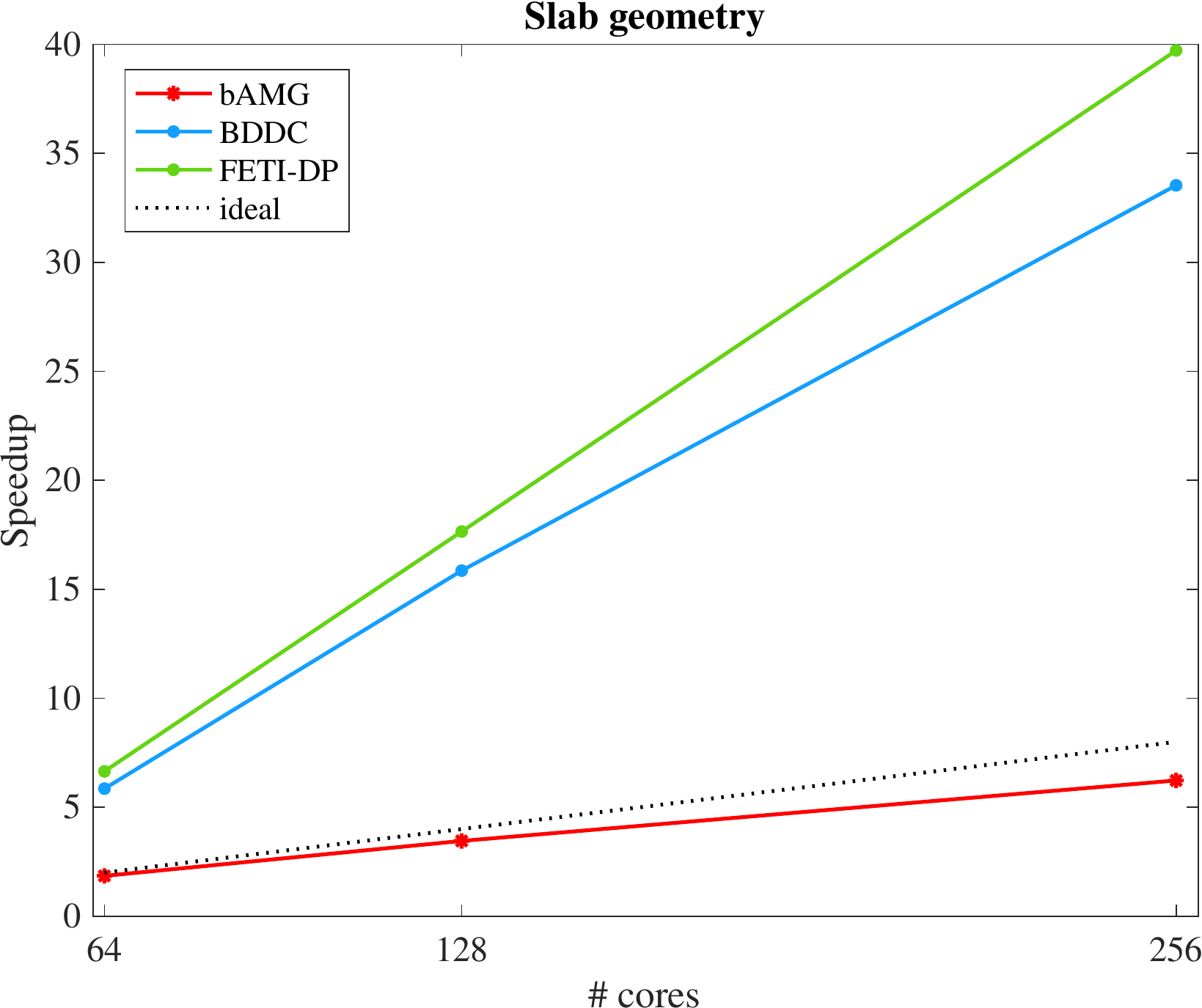}
		\ \ \ \ \  
		\includegraphics[scale=.45]{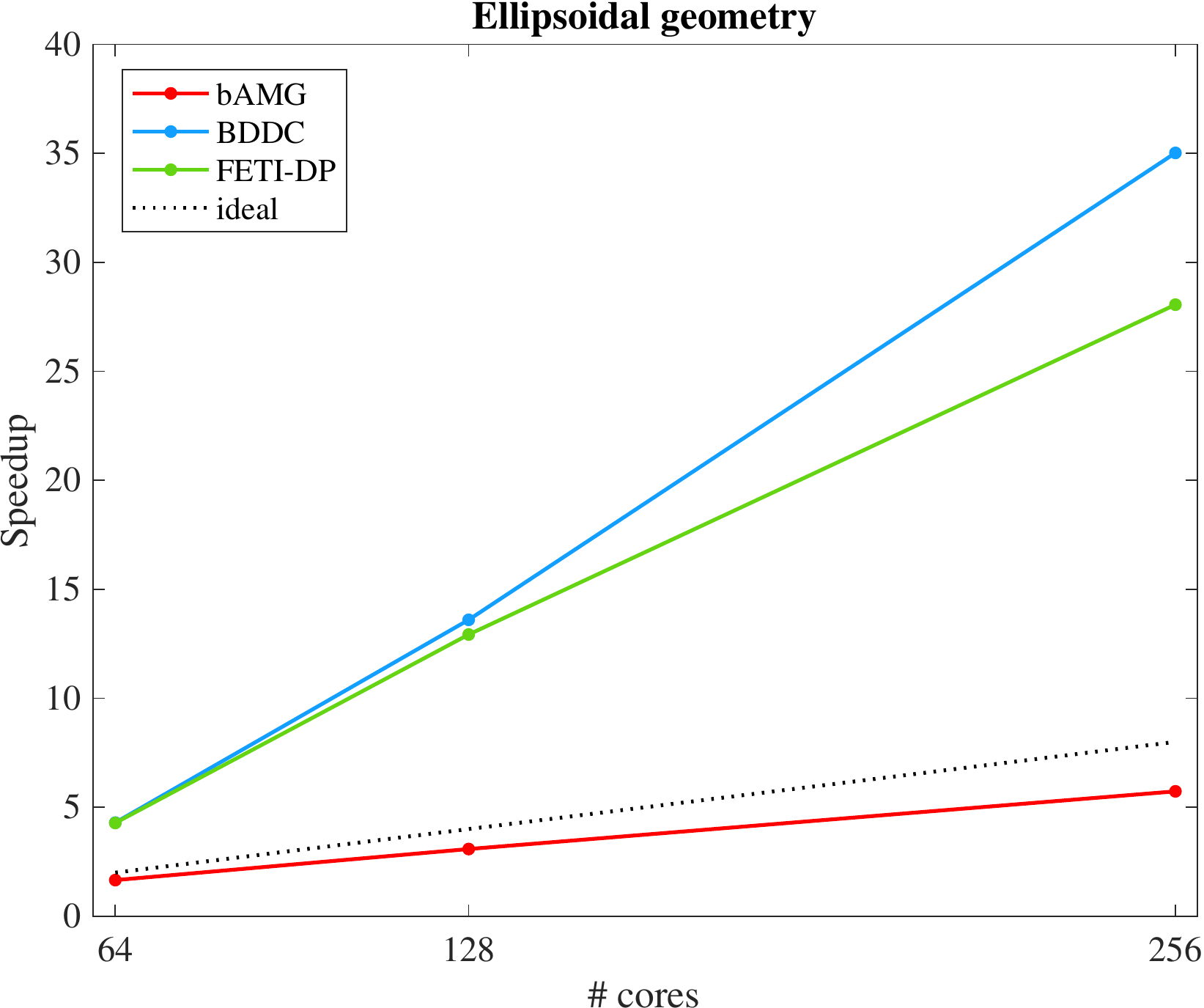}
		\caption{Strong scaling test on the cluster Indaco. Left: slab domain with global mesh $192\cdot 192\cdot 32$. Right: ellipsoidal domain with global mesh $128\cdot 128\cdot 64$. Simulations of 2 ms of cardiac activation with $dt = 0.05$ ms (40 time steps). Comparison of actual parallel speedup computed with respect to 32 cores (ideal speedup dotted). Performed on Indaco cluster.}
		\label{fig_bido_speedup}
	\end{figure}
	
	\begin{table}[h]
		\centering
		\begin{tabular}{*{16}{r}}
			\hline
			\multirow{2}{*}{procs} 		& \multicolumn{5}{c}{bAMG}     	& \multicolumn{5}{c}{BDDC}     	& \multicolumn{5}{c}{FETI-DP}\\
			& nit   & lit   & time	& $S_{32}$  & $S_{64}$		& nit   & lit   & time	& $S_{32}$	& $S_{64}$	& nit   & lit 	& time	& $S_{32}$  & $S_{64}$\\
			\hline
			32     	& 1.25	& 250	& 116.0	& -	   & -	    & 1.0   &27 	&348.2 	& -	    & -		& 1.25	& 11	& 352.2 & -  & -\\
			64      & 1.25	& 252	& 62.7	& 1.8  & -	    & 1.22   &32 	&59.5 	& 5.8   & -		& 1.25	& 17	& 53.0 & 6.6 & -  \\
			128     & 1.25	& 252	& 33.6	& 3.5  & 1.8	& 1.22	   &37 	&21.9 	& 15.8  & 2.7	& 1.25	& 21	& 19.9	& 17.6 & 2.6 \\
			256     & 1.25	& 252	& 18.6	& 6.2  & 3.4    & 1.22   &22 	&10.4 	& 33.5  & 5.7	& 1.25	& 13	& 8.9	& 39.7 & 5.9 \\
			\hline
		\end{tabular}
		\caption{Strong scaling test for the Bidomain decoupled solver on the cluster Indaco.  Slab domain, global mesh $192\cdot 192\cdot 32$ elements, 2,458,434 dofs.	Simulations of 2 ms of cardiac activation with $dt = 0.05$ ms (40 time steps). Comparison of Newton-Krylov solvers with boomerAMG (bAMG), BDDC and FETI-DP preconditioners. Average Newton iterations per time step (nit); average conjugate gradient iterations per Newton iteration (lit); average CPU solution time per time step (time) in seconds; parallel speedup with respect to 32 ($S_{32}$) and 64 ($S_{64}$) processors.}
		\label{table_bido_strong_slab_indaco}
		\vspace{5mm}
		%
		%
		\centering
		\begin{tabular}{*{16}{r}}
			\hline
			\multirow{2}{*}{procs} 		& \multicolumn{5}{c}{bAMG}     	& \multicolumn{5}{c}{BDDC}     	& \multicolumn{5}{c}{FETI-DP}\\
			& nit   & lit   & time	& $S_{32}$  & $S_{64}$		& nit   & lit   & time	& $S_{32}$	& $S_{64}$	& nit   & lit 	& time	& $S_{32}$  & $S_{64}$\\
			\hline
			32     	& 1.92	& 311	& 188.4	& -		& -	    & 1.92	& 36	& 571.8	& -	    & -		& 1.92	& 14	& 558.2	& -     & - \\
			64      & 1.92	& 310	& 113.4	& 1.7   & - 	& 1.92	& 30	& 129.1	& 4.4   & - 	& 1.92	& 19	& 129.7	& 4.3   & - \\
			128     & 1.92	& 310	& 60.5	& 3.1   & 1.9	& 1.92	& 40	& 40.2	& 14.2  & 3.2	& 1.92	& 24	& 42.4	& 13.2  & 3.1 \\
			256     & 1.92	& 311	& 32.2	& 5.8   & 3.1	& 1.92	& 23	& 15.1	& 37.9  & 8.5	& 1.92	& 14	& 19.0	& 29.4  & 6.8 \\
			\hline
		\end{tabular}
		\caption{Strong scaling test for the Bidomain decoupled solver on the cluster Indaco.  Ellipsoidal domain, global mesh $128\cdot 128\cdot 64$ elements, 2,163,330 dofs.	Simulations of 2 ms of cardiac activation with $dt = 0.05$ ms (40 time steps). Comparison of Newton-Krylov solvers with boomerAMG (bAMG), BDDC and FETI-DP preconditioners. Average Newton iterations per time step (nit); average conjugate gradient iterations per Newton iteration (lit); average CPU solution time per time step (time) in seconds; parallel speedup with respect to 32 ($S_{32}$) and 64 ($S_{64}$) processors.}
	\label{table_bido_strong_ell_indaco}
	\end{table}

	\begin{table}[!h]
		\centering
		\begin{tabular}{*{15}{r}}
			\hline
			\multirow{2}{*}{procs}  		&& \multicolumn{5}{c}{RMC}     	&& \multicolumn{5}{c}{LR1}     	\\
			&& nit   & lit   & time	    & $S_{32}$    & $S_{64}$		&& nit   & lit   & time	& $S_{32}$    & $S_{64}$		\\
			\cline{1-1} \cline{3-7} \cline{9-13} 
			32     	&& 1.25	& 16.97	& 220.25	& -		    & -	        && 2.85	& 16.97 	& 502.25 	& -	        & -		\\
			64      && 1.25	& 19.92	& 62.07	    & 3.55      & -	        && 2.85 & 19.57 	& 140.92 	& 3.56      & -		 \\
			128     && 1.25	& 15.3	& 19.2	    & 11.47     & 3.23	    && 2.85	& 15.0 	    & 43.9 	    & 11.44     & 3.21		 \\
			256     && 1.25	& 17.45	& 5.8	    & 37.97     & 10.7      && 2.85 & 29.5 	    & 17.0 	    & 38.08     & 10.68		 \\
			\hline
		\end{tabular}
			\caption{Strong scaling test for the Bidomain decoupled solver on the cluster Indaco.  Slab domain, global mesh $192\cdot 192\cdot 32$ elements, 2,458,434 dofs.	Simulations of 2 ms of cardiac activation with $dt = 0.05$ ms (40 time steps). Comparison of Newton-Krylov solvers with BDDC preconditioner using Rogers-McCulloch (RMC) and Luo-Rudy phase 1 (LR1) ionic models. Average Newton iterations per time step (nit); average conjugate gradient iterations per Newton iteration (lit); average CPU solution time per time step (time) in seconds; parallel speedup ($S_N$) computed with respect to $N=32$ and $N=64$ processors.}
		\label{table_bido_strong_slab_comparison_indaco}
		\vspace{5mm}
		\centering
		\begin{tabular}{*{15}{r}}
			\hline
			\multirow{2}{*}{procs}  		&& \multicolumn{5}{c}{RMC}     	&& \multicolumn{5}{c}{LR1}     	\\
			&& nit   & lit   & time	    & $S_{32}$    & $S_{64}$		&& nit   & lit   & time	& $S_{32}$    & $S_{64}$		\\
			\cline{1-1} \cline{3-7} \cline{9-13} 
			32     	&& 2	& 21.1	& 436.5	& -	            & -		    && 3.95	& 20.5 	& 862.25 	& -		            & -	\\
			64      && 2	& 26.9	& 99.3	& 4.39          & -	        && 3.95 & 25.9 	& 194.87 	& 4.43              & -		 \\
			128     && 2	& 21.4	& 27.27	& 16.0          & 3.64	    && 3.95	& 20.8 	& 53.47 	& 16.12             & 3.64		 \\
			256     && 2	& 30.0	& 8.17	& 53.42         & 12.15     && 3.95 & 29.5 	& 16.08 	& 53.62             & 12.11		 \\
			\hline
		\end{tabular}
		\caption{Strong scaling test for the Bidomain decoupled solver on the cluster Indaco.  Ellipsoidal domain, global mesh $128\cdot 128\cdot 64$ elements, 2,163,330 dofs.	Simulations of 2 ms of cardiac activation with $dt = 0.05$ ms (40 time steps). Comparison of Newton-Krylov solvers with BDDC preconditioner using Rogers-McCulloch (RMC) and Luo-Rudy phase 1 (LR1) ionic models. Average Newton iterations per time step (nit); average conjugate gradient iterations per Newton iteration (lit); average CPU solution time per time step (time) in seconds; parallel speedup ($S_N$) computed with respect to $N=32$ and $N=64$ processors.}
	\label{table_bido_strong_ell_comparison_indaco}
	\end{table}

	\paragraph{\bf Test 3: optimality tests.}
	Tables \ref{table_bido_opt_slab} and \ref{table_bido_opt_ell} report the results of optimality tests, for both slab and ellipsoid geometries, carried on Eos cluster. We fix the number of processors (subdomains) to $4 \cdot 4 \cdot 4$ and we increase the local size $H/h$ from 8 to 24, thus reducing the finite element size $h$. 
	We focus only on the behavior of the BDDC preconditioner, as FETI-DP has been proven to be spectrally equivalent. We consider both scalings ($\rho$-scaling on top, deluxe scaling at the bottom of each table) and we test the solver for increasing primal spaces: V includes only vertex constraints, V+E includes vertex and edge constraints, and V+E+F includes vertex, edge and face constraints. 
	We consider a time interval of 2 ms during the cardiac activation phase. The time step is $dt = 0.05$ ms, for a total amount of 40 time steps. 
	Similar results hold for both geometries. Despite an higher average CPU time when using the deluxe scaling, all the other parameters are quite similar between the two scalings. 
	We observe almost linear dependence of the condition number if the coarsest primal space (i.e. V) is chosen  (see also Figures \ref{fig_bido_opt_slab}, \ref{fig_bido_opt_ell} bottom), while we obtain quasi-optimality if we enrich the primal space by adding edges (V+E) and faces (V+E+F). \\
	
	\begin{table}[h]
		\centering
		\begin{tabular}{*{16}{r}}
			\hline
			\multicolumn{16}{c}{$\rho$-scaling}	\\
			\hline
			\multirow{2}{*}{H/h}	&& \multicolumn{4}{c}{V} 	&& \multicolumn{4}{c}{V+E} 		&& \multicolumn{4}{c}{V+E+F}  \\
			\cline{3-6} \cline{8-11} \cline{13-16}
			&& nlit 	& lit		&time	&cond		&& nlit 	& lit		&time	&cond		&& nlit 	& lit		&time	&cond	\\
			\cline{1-1} \cline{3-6} \cline{8-11} \cline{13-16}
			4			&&1.24 &26 &1.7 &8.4 	 	            &&1.24 &11 &0.9 &1.9 	    &&1.24 &9 &0.9 &1.7 \\
			8			&&1.21 &47 &3.4 &24.1 	                &&1.24 &14 &1.4 &2.6 	    &&1.24 &12 &1.4 &2.5 \\
			12			&&1.04 &66 &10.3 &42.9 	                &&1.17 &18 &6.4 &3.2 	    &&1.21 &15 &4.5 &3.2 \\
			16 			&& \multicolumn{4}{c}{out of memory}    &&1.0 &20 &11.9 &3.7 	    &&1.0 &20 &11.6 &3.7 \\
			20			&& \multicolumn{4}{c}{out of memory}	&&1.0 &22 &34.2 &4.2 	    &&1.0 &20 &32.5 &4.2 \\
			24			&& \multicolumn{4}{c}{out of memory}	&&1.0 &23 &83.1 &4.6 	    &&1.0 &21 &80.5 &4.5 \\
			\hline
			\multicolumn{16}{c}{deluxe scaling}	\\
			\hline
			\multirow{2}{*}{H/h}	&& \multicolumn{4}{c}{V} 	&& \multicolumn{4}{c}{V+E} 		&& \multicolumn{4}{c}{V+E+F}  \\
			\cline{3-6} \cline{8-11} \cline{13-16}
			&& nlit 	& lit		&time	&cond		&& nlit 	& lit		&time	&cond		&& nlit 	& lit		&time	&cond	\\
			\cline{1-1} \cline{3-6} \cline{8-11} \cline{13-16}
			4			&&1.24 &26 &1.9 &8.4 	 	    &&1.24 &11 &0.9 &1.9 	 	    &&1.24 &9 &1.0 &1.7 \\
			8			&&1.24 &47 &4.1 &24.0 	        &&1.24 &14 &1.7 &2.6 	 	    &&1.24 &12 &1.8 &2.5 \\
			12			&&1.07 &65 &14.7 &42.7 	        &&1.17 &18 &5.5 &3.2 	 	    &&1.21 &15 &7.8 &3.2 \\
			16 			&&1.0 &80 &30.0 &63.7 	        &&1.0 &20 &19.1 &3.7 		    &&1.0 &20 &21.4 &3.7 \\
			20			&&1.0 &90 &93.8 &86.3 	        &&1.0 &22 &73.9 &4.2 		    &&1.0 &20 &70.0 &4.2 \\
			24			&&1.0 &99 &211.9 &110.1 	    &&1.0 &24 &205.8 &4.5 	        &&1.0 &21 &247.3 &4.5 \\
			\hline
		\end{tabular}
	\caption{Optimality test on Eos cluster for PCG - BDDC.
	Slab domain, $4 \cdot 4 \cdot 4$ subdomains, increasing local size from $4 \cdot 4 \cdot 4$ to $24 \cdot 24 \cdot 24$. Comparison between different scaling and different primal sets (V = vertices, E = edges, F = faces). Average non-linear iterations (nlit), average number of linear iteration, average CPU time in seconds and average condition number per time step.}	
\label{table_bido_opt_slab}
	\end{table}
	
	\begin{figure}[!h]
		\centering
		\includegraphics[scale=.45]{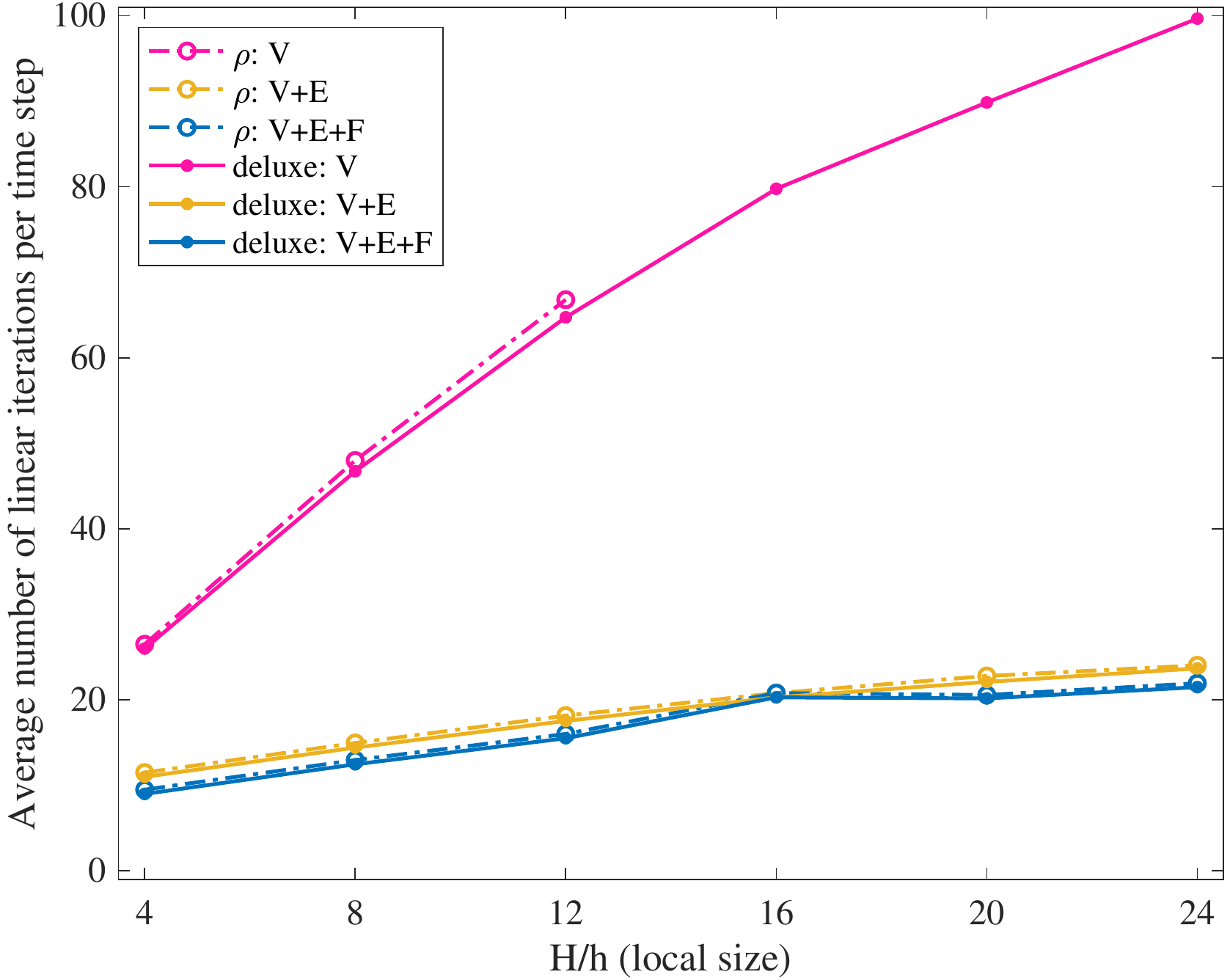}
		\ \ \ \ \  
		\includegraphics[scale=.45]{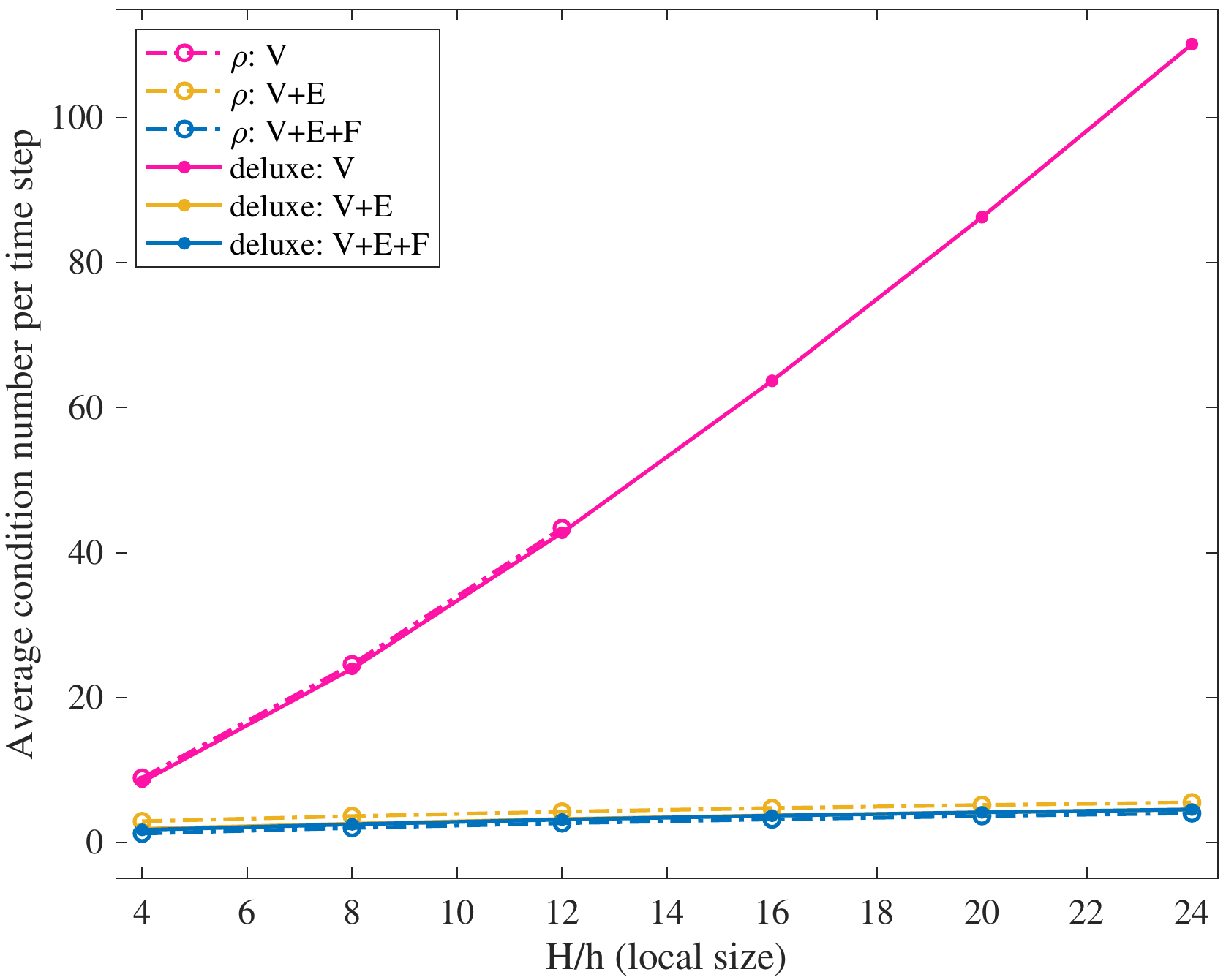}
		\caption{Optimality test on Eos cluster for PCG - BDDC.
			Slab domain, $4 \cdot 4 \cdot 4$ subdomains, increasing local size from $4 \cdot 4 \cdot 4$ to $24 \cdot 24 \cdot 24$. Comparison between different scaling (dash-dotted $\rho$-scaling, continuous {\em deluxe} scaling) and different primal sets (V = vertices, E = edges, F = faces). Average number of linear iterations (left) and average condition numbers (right) per time step.}
		\label{fig_bido_opt_slab}
	\end{figure}
	
	\begin{table}[h]
		\centering
		\begin{tabular}{*{16}{r}}
			\hline
			\multicolumn{16}{c}{$\rho$-scaling}	\\
			\hline
			\multirow{2}{*}{H/h}	&& \multicolumn{4}{c}{V} 	&& \multicolumn{4}{c}{V+E} 		&& \multicolumn{4}{c}{V+E+F}  \\
			\cline{3-6} \cline{8-11} \cline{13-16}
			&& nlit 	& lit		&time	&cond		&& nlit 	& lit		&time	&cond		&& nlit 	& lit		&time	&cond	\\
			\cline{1-1} \cline{3-6} \cline{8-11} \cline{13-16}
			8			&&2.0 &50 	    &5.5 	&30.1 	  &&2.0 &17 &2.4 &4.3 	 	&&2.0 &16 &2.4   &4.0 \\
			12			&&2.0 &66 	    &11.8   &54.6 	  &&2.0 &19 &5.6 &5.4 	    &&2.0 &18 &5.6   &4.9 \\
			16 			&&2.0 &80 	    &33.9   &80.6 	  &&2.0 &21 &16.9 &6.3 	    &&2.0 &21 &16.9 &5.7 \\
			20			&&2.0 &91 	    &90.3   &108.2    &&2.0 &22 &44.7 &6.9 	    &&2.0 &21 &44.9 &6.3 \\
			24			&&1.46 &100     &206.3  &137.2 	  &&1.46 &24 &109.3 &7.5 	&&1.46 &23 &84.0 &6.8 \\
			\hline
			\multicolumn{16}{c}{deluxe scaling}	\\
			\hline
			\multirow{2}{*}{H/h}	&& \multicolumn{4}{c}{V} 	&& \multicolumn{4}{c}{V+E} 		&& \multicolumn{4}{c}{V+E+F}  \\
			\cline{3-6} \cline{8-11} \cline{13-16}
			&& nlit 	& lit		&time	&cond		&& nlit 	& lit		&time	&cond		&& nlit 	& lit		&time	&cond	\\
			\cline{1-1} \cline{3-6} \cline{8-11} \cline{13-16}
			8			&&2.0 &49 &6.7 	&31.0 	 	    &&2.0 &17 &3.1 &4.3 	 	&&2.0 &16 &3.0 &4.0 \\
			12			&&2.0 &54 &27.2 	&54.6 	    &&2.0 &19 &9.6 &5.4 	    &&2.0 &18 &9.6 &4.9 \\
			16 			&&2.0 &79 &54.1 	&80.6 	    &&2.0 &21 &32.1 &6.2 	    &&2.0 &21 &32.7 &5.7 \\
			20			&&2.0 &90 &142.4  &108.2 	    &&2.0 &22 &125.8 &7.0 	    &&2.0 &21 &111.1 &6.3 \\
			24			&&1.46 &99 &329.3  &137.1 	    &&1.46 &24 &236.7 &7.5 	    &&1.46 &22 &247.1 &6.8 \\
			\hline
		\end{tabular}
	\caption{Optimality test on Eos cluster for PCG - BDDC.
	Ellipsoidal domain, $4 \cdot 4 \cdot 4$ subdomains,  increasing local size from $8 \cdot 8 \cdot 8$ to $24 \cdot 24 \cdot 24$. Comparison between different scaling and different primal sets (V = vertices, E = edges, F = faces). Average non-linear iterations (nlit), average number of linear iteration, average CPU time in seconds and average condition number per time step.}	
\label{table_bido_opt_ell}
	\end{table}
	
	\begin{figure}[!h]
		\centering
		\includegraphics[scale=.45]{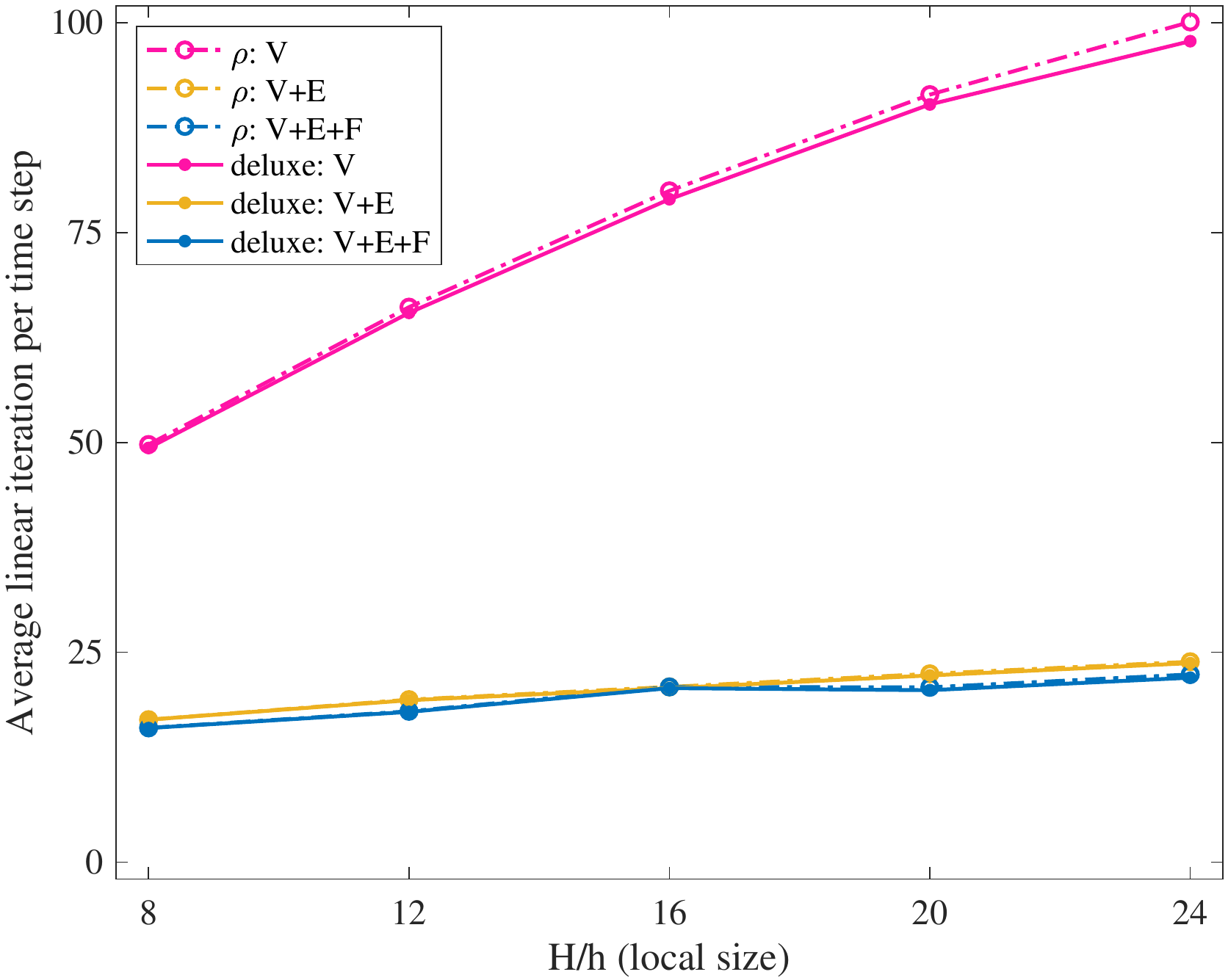}
		\ \ \ \ \  
		\includegraphics[scale=.45]{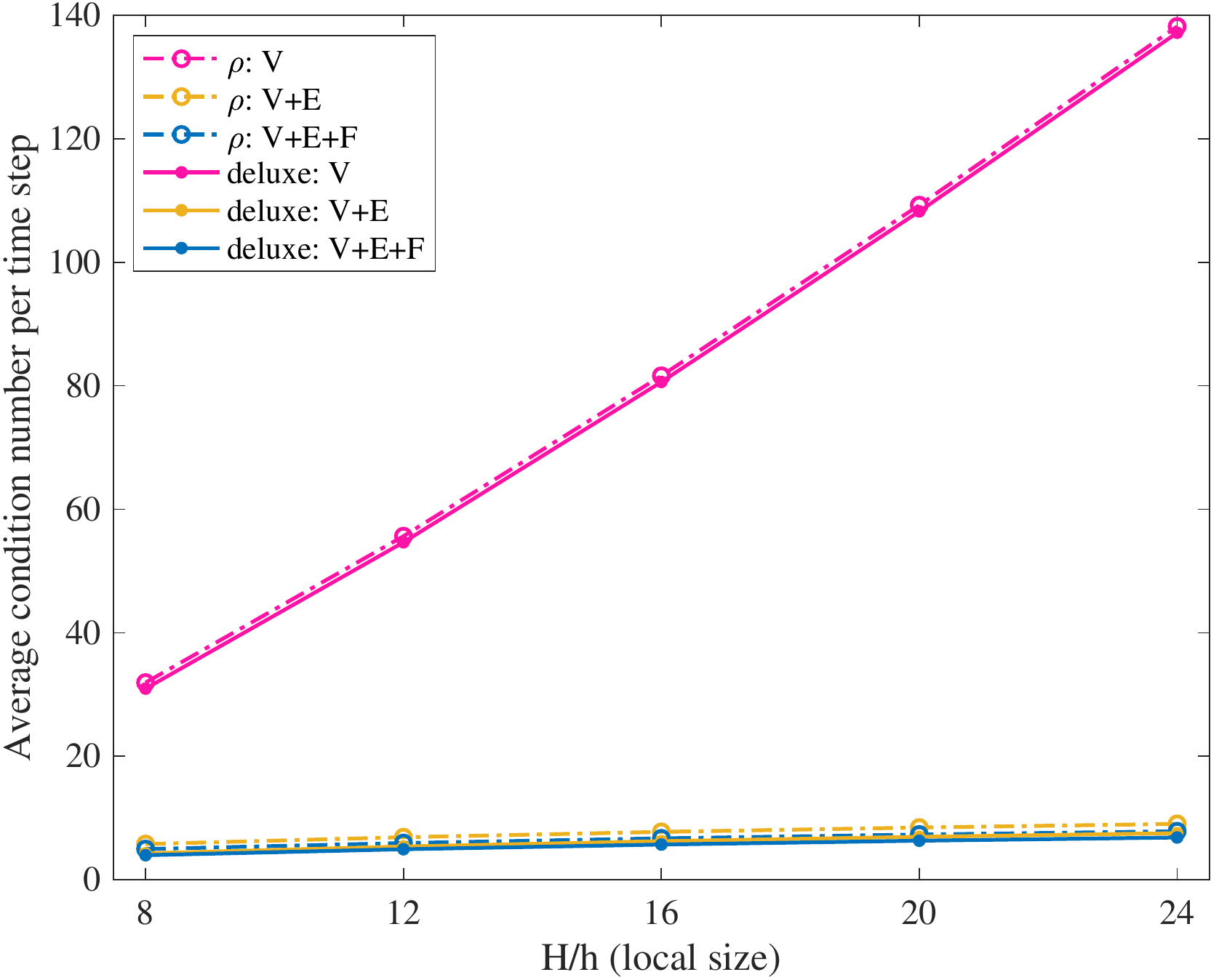}
		\caption{Optimality test on Eos cluster for PCG - BDDC with different scaling (dash-dotted $\rho$-scaling, continuous {\em deluxe} scaling) and primal sets (V = vertices, E = edges, F = faces). Ellipsoidal domain, $4 \cdot 4 \cdot 4$ subdomains, increasing local size from $8 \cdot 8 \cdot 8$ to $24 \cdot 24 \cdot 24$.  Average number of linear iterations (left) and average condition numbers (right) per time step.}
		\label{fig_bido_opt_ell}
	\end{figure}

	\paragraph{\bf Test 4: whole beat (activation - recovery) simulations.} 
	In this last set of tests (performed on the Indaco cluster), we compare the performance of our dual-primal and multigrid preconditioners during a whole beat, i.e. during a complete activation - recovery interval over the computational domain.
	We fix the number of subdomains to $128=8 \cdot 8 \cdot 2$ and the global mesh size to $192 \cdot 96 \cdot 24$, obtaining local problems with 8,450 dofs.  
	We consider either a portion of ellipsoid, defined by $\varphi_{\min} = -\pi/2$, $\varphi_{\max} = 0$, $\theta_{\min} = -3/8 \pi$ and $\theta_{\max} = \pi/8$, or a slab of dimensions $1.92 \times 1.92 \times 0.48$ cm$^3$. The ellipsoid test is on a time interval of $[0,170]$ ms, for a total of 3400 time steps, while the slab test is on the time interval $[0, 120]$ ms, for a total of 2400 time steps. Both time intervals are
	enough to complete the activation and recovery phases over the whole domains, given the short action potential duration of the RMC ionic model considered.
	In Figures \ref{fig_bido_full_ell_indaco} and \ref{fig_bido_full_slab_indaco} we report the trend of the average number of linear iteration per time step during the simulation. The number of iterations remains bounded and almost constant during the test. Moreover, we notice a huge difference between the multigrid  and the dual-primal preconditioners, with a reduction of more than $90\%$ for the latter. 
	If we focus on the trend of the dual-primal preconditioners' average number of linear iterations (Figures \ref{fig_bido_full_ell_indaco} and \ref{fig_bido_full_slab_indaco}, on the right), we see that on both domains FETI-DP is affected by the different phases of the action potential: there is an initial peak during the activation phase, followed by an increase in the number of linear iterations as the potential propagates in the tissue and by a slow decrease as wider portions of tissue return to resting. Similar behavior can be observed for the BDDC preconditioner on the slab domain: there is an initial peak corresponding to the activation phase, followed by a constant period, as the tissue turn to resting. This trend is not visible for BDDC on the ellipsoidal domain, due to the complexity of the geometry.
	We also observe a better performance of the dual-primal preconditioners in terms of average CPU time per time step (see Table \ref{table_bido_full_indaco}). 
	
	\begin{table}[!h]
		\centering
		\begin{tabular}{*{16}{c}}
			\cline{2-16}
			& \multirow{2}{*}{procs}	&& \multirow{2}{*}{dofs} && \multicolumn{3}{c}{bAMG} 	&& \multicolumn{3}{c}{BDDC} 		&& \multicolumn{3}{c}{FETI-DP}  \\
			&&&  && nlit 	& lit		&time	&& nlit 	& lit		&time		&& nlit 	& lit		&time	\\
			\cline{1-2} \cline{4-4} 	\cline{6-8} \cline{10-12} \cline{14-16}
			slab		& 128 		&& 8,450	&& 1.4 	  &185 	  &11.28 	&&1.4 	&19  	& 8.02	&&1.4 	&12		&7.62 \\
			ellipsoid 	& 128		&& 8,450 	&& 1.97   &328 	  &13.24	&& 1.97 & 30 	& 8.85	&& 1.97 & 21	& 8.05 	\\
			\hline
		\end{tabular}
	\caption{Whole beat simulation on time interval $[0,170]$ ms, 3400 time steps for the ellipsoidal domain and time interval $[0,120]$ ms, 2400 time steps for the slab. Fixed number of subdomains $8 \cdot 8 \cdot 2$ and fixed global mesh $ 192 \cdot 96 \cdot 24$. Comparison of average Newton steps, average linear iterations and average CPU time (in sec.) per time step.}	
\label{table_bido_full_indaco}
	\end{table}
	
	\begin{figure}[!h]
		\centering
		\includegraphics[scale=.45]{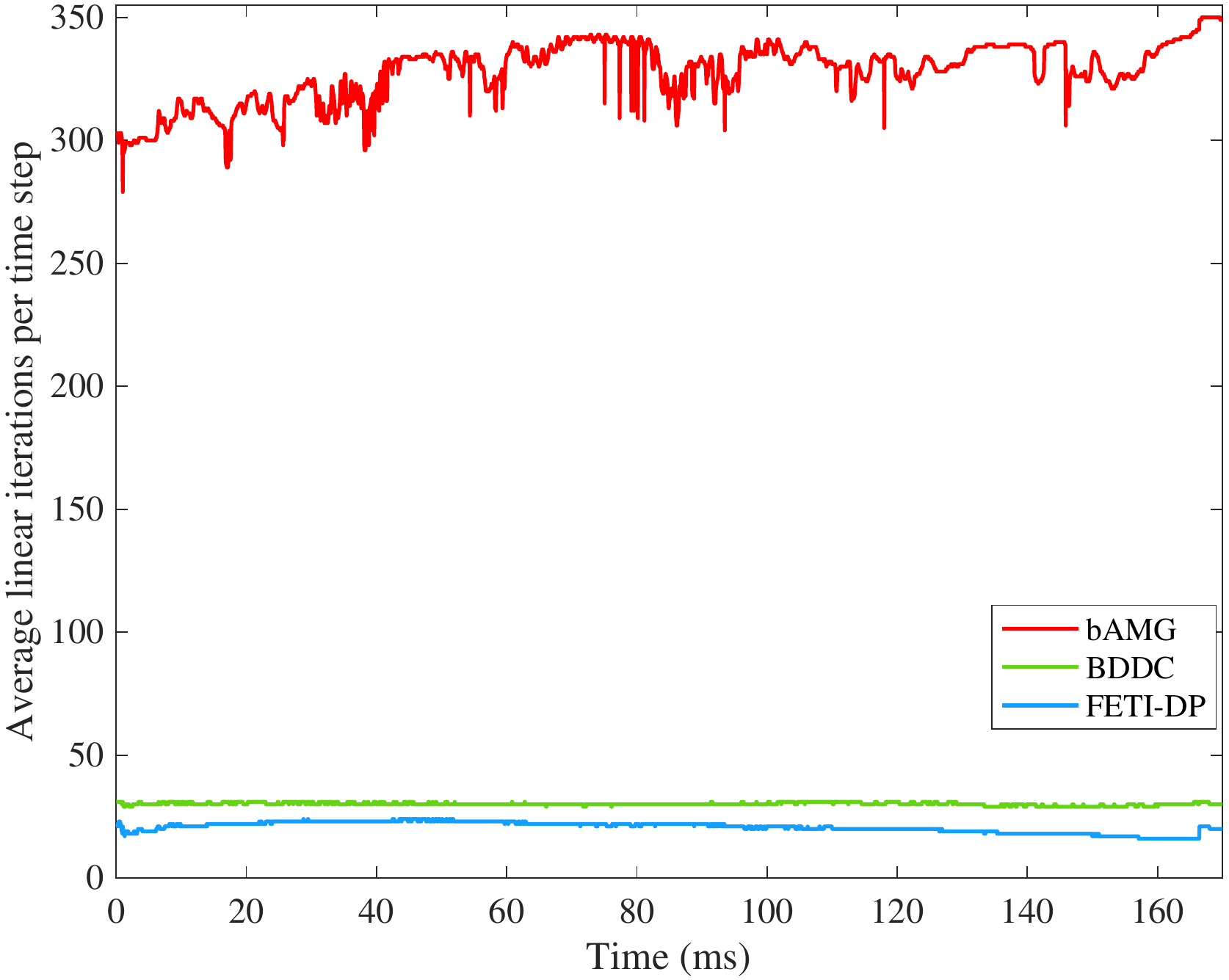}
		\ \ \ \ \  
		\includegraphics[scale=.45]{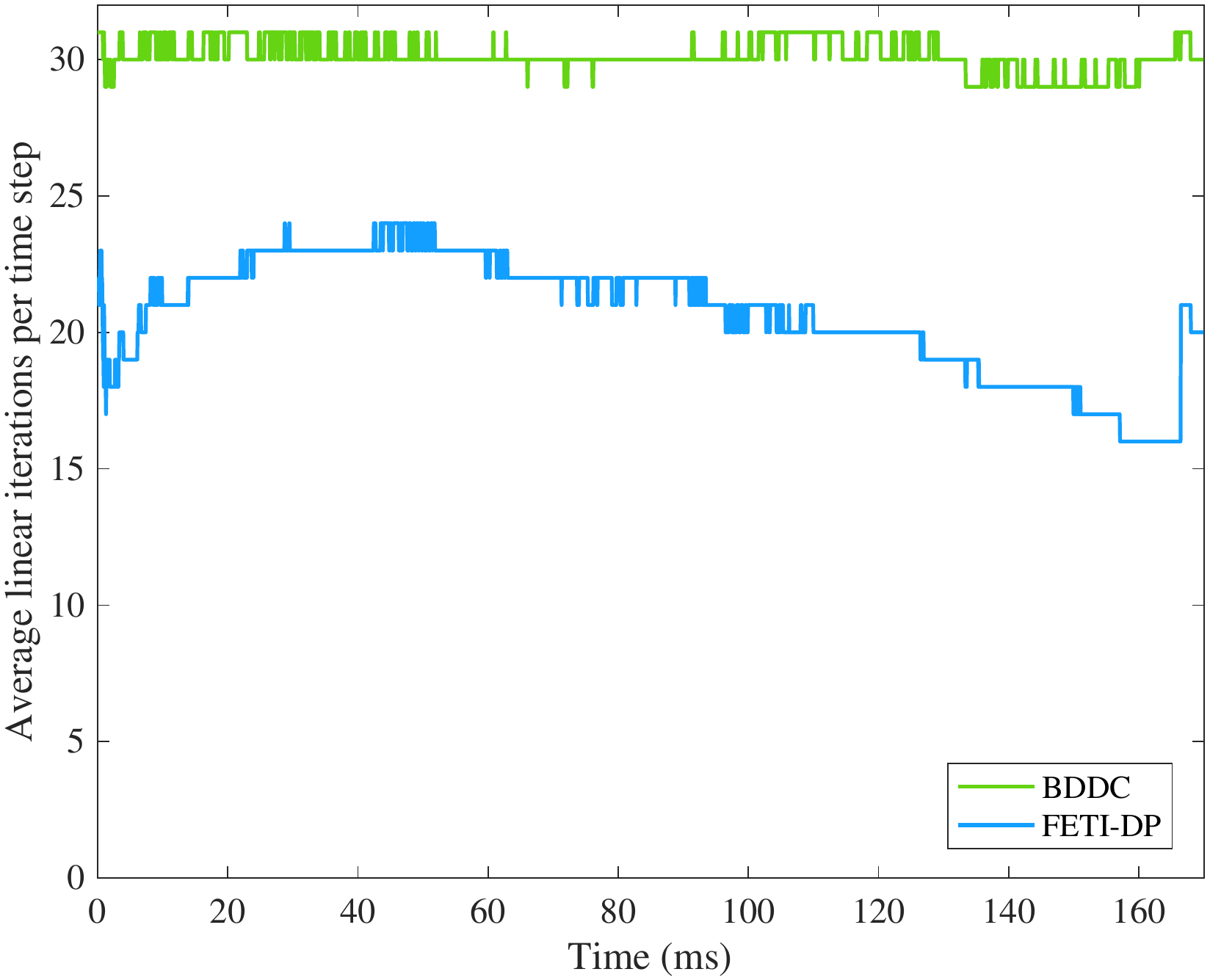}
		\caption{Whole beat simulation on ellipsoidal domain, time interval $[0,170]$ms, 3400 time steps. Fixed number of subdomains $8 \cdot 8 \cdot 2$ and fixed global mesh $ 192 \cdot 96 \cdot 24$. Average number of linear iterations per time step (left), zoom over dual-primal preconditioner (right).}
		\label{fig_bido_full_ell_indaco}
	\end{figure}
	
	\begin{figure}[!h]
		\centering
		\includegraphics[scale=.45]{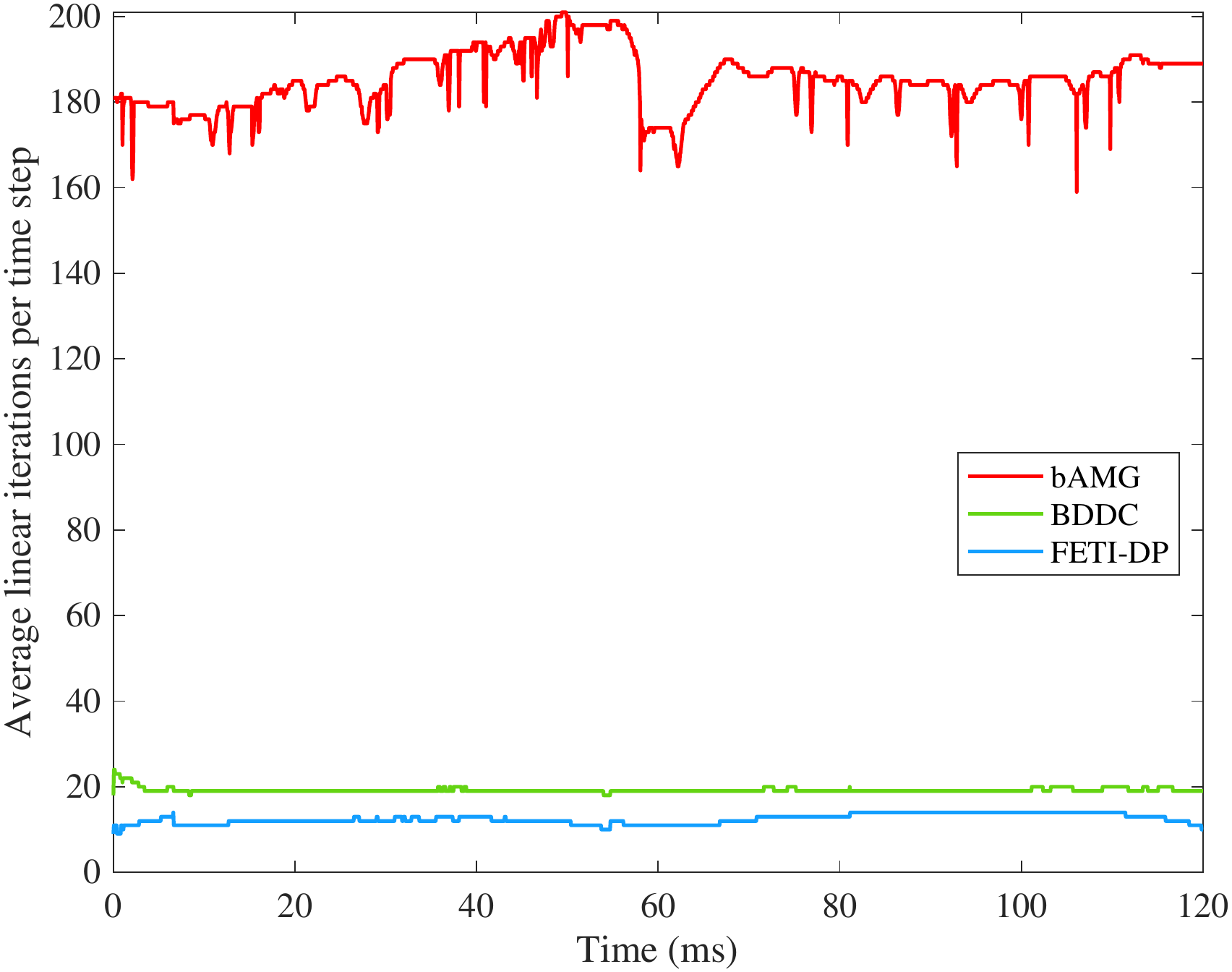}
		\ \ 
		\includegraphics[scale=.45]{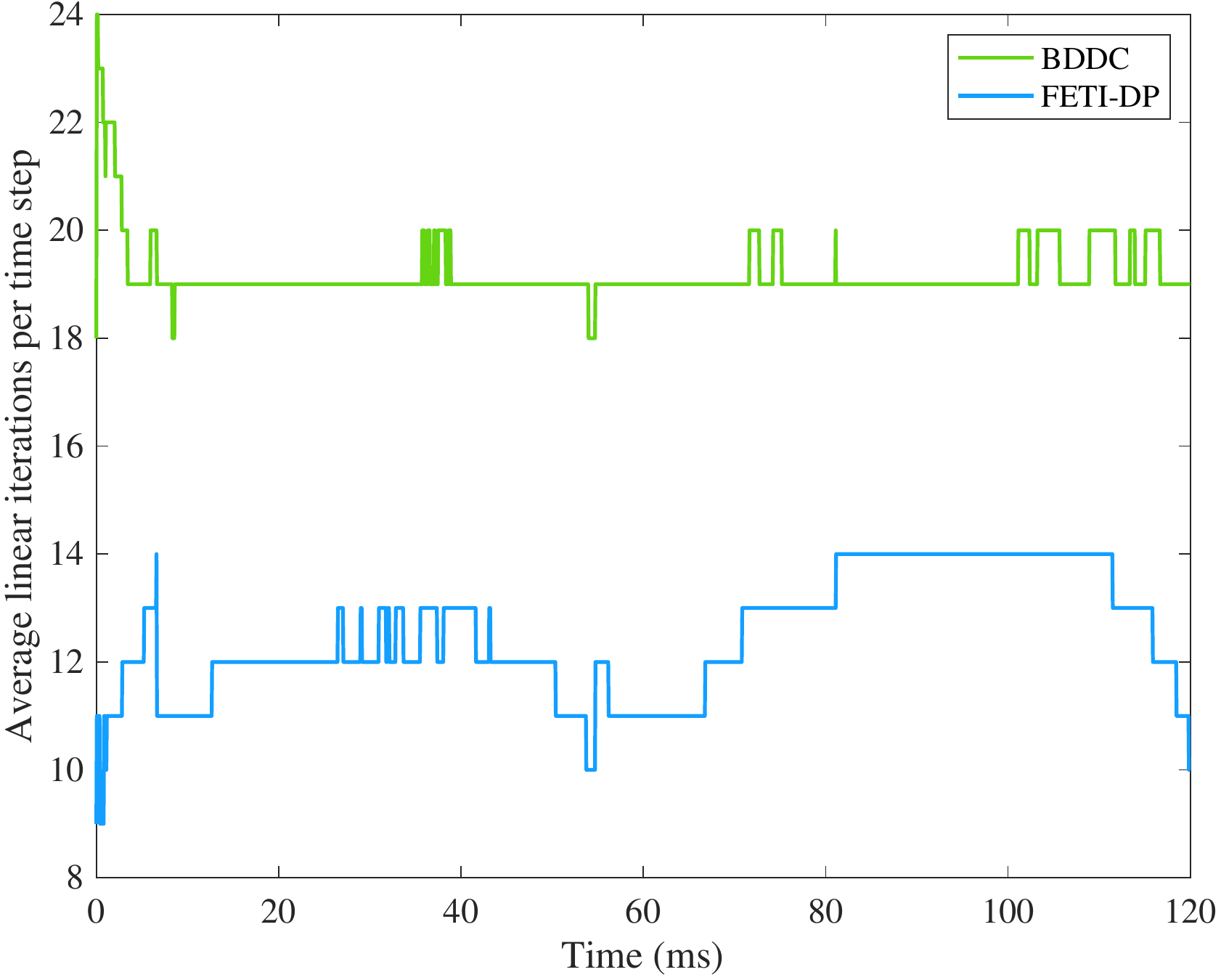}
		\caption{Whole beat simulation on slab domain, time interval $[0,120]$ms, 2400 time steps. Fixed number of subdomains $8 \cdot 8 \cdot 2$ and fixed global mesh $ 192 \cdot 96 \cdot 24$. Average number of linear iterations per time step (left), zoom over dual-primal preconditioner (right).}
		\label{fig_bido_full_slab_indaco}
	\end{figure}

	\section{Conclusions}				
	We have constructed dual-primal preconditioners for fully implicit discretizations of the Bidomain system, which are solved through a decoupling strategy. We have proved a 
	convergence bound of the preconditioned FETI-DP and BDDC Bidomain operators with deluxe scaling. Parallel numerical tests validate the bound and show the efficiency and robustness of the solver, thus enlarging the class of methods available for the efficient and accurate numerical solution of this complex biophysical reaction-diffusion model. 
	Additional research is needed in order to asses the performance of the proposed dual-primal solvers for more realistic ionic models and with respect to optimized multigrid solvers.

\end{document}